\documentclass[11pt,onecolumn]{article}

\usepackage{a4wide}
\usepackage{amsmath,amsthm,epsfig,amssymb,amsbsy}
\usepackage{enumerate}
\usepackage{comment}
\usepackage{algorithm}
\usepackage{algorithmic}
\usepackage{amsfonts}       
\usepackage{dsfont}
\usepackage{enumitem}
\usepackage{amssymb}
\usepackage{comment}
\usepackage{accents}

\usepackage{pgfplots}
\pgfplotsset{compat=1.18}
\usepackage{subfig}

\usepackage[
giveninits=true,
maxbibnames=9,
maxcitenames=2,
backend=bibtex,
bibstyle=numeric,
doi=false,isbn=false,url=false,
]{biblatex}
\renewbibmacro{in:}{
    \ifentrytype{article}
    {}
    {\printtext{\bibstring{in}\intitlepunct}}}
\newenvironment{keywords}{
\begin{paragraph}{Keywords:}
}
{
\end{paragraph}
}
    
\newenvironment{subclass}{
\begin{paragraph}{AMS Subject Classification:}
}
{\end{paragraph}
}

\DeclareMathOperator*{\argmin}{arg\,min}
\DeclareMathOperator*{\argminimax}{arg\,minimax}
\DeclareMathOperator*{\argmax}{arg\,max}

\DeclareMathOperator{\infconv}{\mathbin{\square}}

\DeclareMathOperator{\logsumexp}{LogSumExp}

\DeclareMathOperator{\vecmax}{vecmax}
\DeclareMathOperator{\Li}{Li}
\DeclareMathOperator{\gph}{gph}
\DeclareMathOperator{\sumexp}{SumExp}

\DeclareMathOperator{\dom}{dom}

\DeclareMathOperator{\intr}{int}
\DeclareMathOperator{\bdry}{bdry}

\DeclareMathOperator{\cl}{cl}

\DeclareMathOperator{\ran}{rge}

\DeclareMathOperator{\proj}{proj}

\DeclareMathOperator{\zer}{zer}
\DeclareMathOperator{\dist}{dist}

\newcommand{\bR}{\mathbb{R}}

\newcommand{\exR}{\overline{\mathbb{R}}}

\newcommand{\cC}{\mathcal{C}}

\newcommand{\cB}{\mathcal{B}}

\makeatletter
\newcommand{\aprox}[3][\@nil]{%
  \def\tmp{#1}%
   \ifx\tmp\@nnil
       \operatorname{aprox}_{#3}^{#2}
    \else
         \operatorname{aprox}_{#3}^{#1 \star #2}
    \fi}

\newcommand{\bprox}[3][\@nil]{%
  \def\tmp{#1}%
   \ifx\tmp\@nnil
       \operatorname{bprox}_{#3}^{#2}
    \else
        \operatorname{bprox}_{#3}^{#1 #2}
    \fi}

\makeatother

\usepackage{hyperref}
\hypersetup{
    colorlinks=true,
    linkcolor=blue,
    filecolor=magenta,
    citecolor =magenta,  
    urlcolor=magenta,
    pdftitle={Anisotropic Proximal Gradient}
    }
\usepackage[capitalize]{cleveref}[0.19]
\usepackage{aliascnt}

\crefname{section}{section}{sections}
\crefname{subsection}{subsection}{subsections}
\Crefname{section}{Section}{Sections}
\Crefname{subsection}{Subsection}{Subsections}

\Crefname{figure}{Figure}{Figures}

\crefformat{equation}{\textup{#2(#1)#3}}
\crefrangeformat{equation}{\textup{#3(#1)#4--#5(#2)#6}}
\crefmultiformat{equation}{\textup{#2(#1)#3}}{ and \textup{#2(#1)#3}}
{, \textup{#2(#1)#3}}{, and \textup{#2(#1)#3}}
\crefrangemultiformat{equation}{\textup{#3(#1)#4--#5(#2)#6}}%
{ and \textup{#3(#1)#4--#5(#2)#6}}{, \textup{#3(#1)#4--#5(#2)#6}}{, and \textup{#3(#1)#4--#5(#2)#6}}

\Crefformat{equation}{#2Equation~\textup{(#1)}#3}
\Crefrangeformat{equation}{Equations~\textup{#3(#1)#4--#5(#2)#6}}
\Crefmultiformat{equation}{Equations~\textup{#2(#1)#3}}{ and \textup{#2(#1)#3}}
{, \textup{#2(#1)#3}}{, and \textup{#2(#1)#3}}
\Crefrangemultiformat{equation}{Equations~\textup{#3(#1)#4--#5(#2)#6}}%
{ and \textup{#3(#1)#4--#5(#2)#6}}{, \textup{#3(#1)#4--#5(#2)#6}}{, and \textup{#3(#1)#4--#5(#2)#6}}

\newtheorem{theorem}{Theorem}[section]

\newlist{thmenum}{enumerate}{1} 
\setlist[thmenum]{label=(\roman*), ref=\thetheorem(\roman*), font=\rm} 
\crefalias{thmenumi}{theorem}

\newaliascnt{corollary}{theorem}

\aliascntresetthe{corollary}

\newaliascnt{lemma}{theorem}
\newtheorem{lemma}[lemma]{Lemma}
\aliascntresetthe{lemma}

\newlist{lemenum}{enumerate}{1}
\setlist[lemenum]{label=(\roman*), ref=\thelemma(\roman*), font=\rm} 
\crefalias{lemenumi}{lemma}

\newaliascnt{proposition}{theorem}
\newtheorem{proposition}[proposition]{Proposition}
\aliascntresetthe{proposition}

\newlist{propenum}{enumerate}{1} 
\setlist[propenum]{label=(\roman*), ref=\theproposition(\roman*), font=\rm}
\crefalias{propenumi}{proposition}

\newaliascnt{definition}{theorem}
\newtheorem{definition}[definition]{Definition}
\aliascntresetthe{definition}

\newlist{defenum}{enumerate}{1}
\setlist[defenum]{label=(\roman*), ref=\thedefinition(\roman*), font=\rm}
\crefalias{defenumi}{definition}

\newlist{corenum}{enumerate}{1} 
\setlist[corenum]{label=(\roman*), ref=\thecorollary(\roman*), font=\rm} 
\crefalias{corenumi}{corollary}

\theoremstyle{remark}
\newaliascnt{remark}{theorem}
\newtheorem{remark}[remark]{Remark}
\aliascntresetthe{remark}

\newaliascnt{example}{theorem}
\newtheorem{example}[example]{Example}
\aliascntresetthe{example}

\title{All roads lead to Rome: \\
Path-following Augmented Lagrangian Methods \\
via Bregman Proximal Regularization}
\author{Emanuel Laude\thanks{	Proxima Fusion GmbH,
		Fl\"o\ss ergasse 2, 81369 Munich, Germany~
		{\tt%
			\href{mailto:elaude@proximafusion.com}{elaude@proximafusion.com}%
		}
	}}

\begin{document}

\maketitle
\begin{abstract}
We study Bregman proximal augmented Lagrangian methods with second-order oracles for convex convex-composite optimization problems. 
The outer loop is an instance of the \emph{Bregman proximal point algorithm} with relative errors in the sense of Solodov and Svaiter, applied to the KKT operator associated with the problem.
Akin to classical Lagrange--Newton methods, including primal-dual interior point methods the Bregman proximal point algorithm repeatedly solves regularized KKT inclusions by minimizing a smooth Bregman augmented Lagrangian function, obtained after marginalizing out the multiplier variables. Thanks to non-Euclidean geometries the marginal function is generalized self-concordant and therefore within the regime of Newton's method which converges quadratically if the step-size in the outer proximal point loop is chosen carefully.
The operator-theoretic viewpoint allows us to employ the framework of metric subregularity to derive fast rates for the outer loop, and eventually state a joint complexity bound.
Important special cases of our framework are a proximal variant of the \emph{exponential multiplier method} due to Tseng and Bertsekas and interior-point proximal augmented Lagrangian schemes closely related to those of Pougkakiotis and Gondzio.
\end{abstract}

\begin{subclass}
65K05 $\cdot$ 49J52 $\cdot$ 90C30
\end{subclass}
\begin{subclass}
65K05 $\cdot$ 49J52 $\cdot$ 90C30
\end{subclass}
\tableofcontents

\section{Introduction}
\subsection{Motivation}
Augmented Lagrangian methods are among the most fundamental techniques for constrained convex optimization.
In the classical Euclidean setting, the augmented Lagrangian method admits an interpretation as the proximal point algorithm applied to the Lagrange dual function \cite{rockafellar1976augmented}.
While this viewpoint is useful, it is not the perspective adopted in this paper.

Instead, we focus on the \emph{proximal augmented Lagrangian method} introduced in the same paper \cite{rockafellar1976augmented}, which applies the proximal point algorithm directly to the primal-dual KKT operator associated with the problem. 
When combined with a Newton oracle for solving the resulting regularized subproblems, proximal augmented Lagrangian methods have recently been shown to achieve state-of-the-art performance for quadratic programming \cite{hermans2019qpalm,schwan2023piqp}.

The proximal augmented Lagrangian method naturally extends beyond Euclidean geometry through the Bregman proximal point algorithm \cite{censor1992proximal,chen1993convergence,Eckstein93,eckstein1998approximate,burachik1998generalized,solodov2000inexact}, leading to Bregman proximal augmented Lagrangian methods, possibly allowing for relative errors in the sense of Solodov and Svaiter \cite{solodov2000inexact}.

A key observation motivating this work is that primal-dual interior point methods and proximal augmented Lagrangian methods equipped with a Newton oracle can be understood within a common framework.
Both classes of methods aim to compute a zero of the KKT operator by repeatedly solving regularized KKT inclusions using Newton’s method.
Interior point methods introduce regularization inside the KKT equations via barrier or central-path perturbations and apply Newton’s method to the resulting nonlinear system.
In contrast, proximal augmented Lagrangian methods regularize the KKT operator itself via a proximal term. Newton’s method is then applied to the corresponding resolvent equation.

This interpretation places both approaches within a broader class of regularized Lagrange–Newton methods, differing primarily in where and how regularization is imposed.
Closely related ideas appear in the interior-point proximal multiplier methods of Pougkakiotis and Gondzio \cite{pougkakiotis2021interior,pougkakiotis2022interior}, as well as in Newton proximal augmented Lagrangian solvers such as \cite{hermans2019qpalm}.
The present work extends and systematizes this connection via a non-Euclidean, Bregman setup:

By formulating the method as a Bregman proximal point algorithm applied to the KKT operator, we can marginalize out the dual variables whenever the Bregman distance separates between decision and multiplier variables and the dual subproblem is tractable.
This reduces the inner task to the minimization of a smooth Bregman augmented Lagrangian with proximal regularization. Thanks to the non-Euclidean geometry the resulting subproblems exhibit generalized self-concordance \cite{nesterov1994interior,bach2010self,sun2019generalized,doikov2025minimizing}. This places the subproblems in the regime of Newton's method for which we adopt a path-following strategy implemented through the proximal step-sizes:
To this end it is proved that the Solodov--Svaiter relative stopping criterion of the inexact Bregman proximal point algorithm, which is formulated in terms of Bregman distances, can be locally certified using the Newton decrement, i.e., the (local) norm of the gradient induced by the Hessian.
By choosing the proximal step-sizes so that each inner iterate starts inside the quadratic convergence region, a small and predictable number of Newton steps suffices to satisfy the Solodov--Svaiter stopping condition. 
Moreover, the operator-theoretic formulation suggests the use of metric subregularity and Hoffman-type error bounds for the KKT operator, which we exploit to derive explicit outer loop convergence rates and, in combination with the Newton analysis, joint complexity bounds for the overall method.
While such error-bound techniques are well established for the classical (Euclidean) proximal point algorithm \cite{luque1984asymptotic,rockafellar2021advances}, their extension to Bregman proximal methods requires a suitable local equivalence between Bregman distances and Euclidean distances which can be guaranteed under \emph{very strict convexity} \cite{bauschke2000dykstras} of the distance generating function.

We highlight two important special cases of the proposed framework.
In the case of entropic dual regularization, the proposed scheme yields a proximal variant of the \emph{exponential multiplier method} of Tseng and Bertsekas \cite{tseng1993convergence}.
When the primal variables are regularized with Burg’s entropy, the resulting algorithm recovers interior-point proximal augmented Lagrangian schemes closely related to those of \cite{pougkakiotis2021interior,pougkakiotis2022interior}, which underpin state-of-the-art solvers for quadratic programming such as \cite{schwan2023piqp}.
\subsection{Contributions}
The contributions of this paper can be summarized as follows:
\begin{enumerate}
    \item
    We extend the convergence theory of inexact Bregman proximal point methods with relative errors \cite{solodov2000inexact} by means of a complexity analysis using metric subregularity of the KKT operator and less restrictive assumptions on the distance generating function.
    To overcome a domain interiority assumption we also provide an ergodic analysis of the scheme utilizing a restricted gap function based on Fitzpatrick's function.

    \item 
    We refine the analysis of the Bregman proximal point algorithm when applied to the KKT operator showing dual subsequential convergence without domain interiority and a finer ergodic analysis utilizing the primal-dual gap function.

    \item
    Via infimal marginalization with respect to the dual variables, the regularized inclusion reduces to the minimization of a smooth Bregman augmented Lagrangian that is quasi self-concordant or self-concordant, depending on the geometry.
We propose a path-following strategy based on proximal step-sizes that aligns the relative Solodov--Svaiter stopping criterion with the Newton decrement, which allows us to analyze a Newton oracle for these subproblems under generalized self-concordance assumptions.
\end{enumerate}

\subsection{Notation}
Let $X$ be a Euclidean space. Let $\langle \cdot,\cdot \rangle : X \times X \to \bR$ be the Euclidean inner product and denote by $\|x\|=\sqrt{\langle x,x \rangle}$ the Euclidean norm. We denote by $\Gamma_0(X)$ the space of all proper, lower semi-continuous (lsc) convex functions $f:X \to \exR := \bR \cup \{+\infty\}$.
We say that $f\in \Gamma_0(X)$ is super-coercive if $f(x)/\|x\|\to \infty$ whenever $\|x\|\to\infty$. 
We introduce the epi-scaling $\tau \star f$ of $f$ defined as $(\tau \star f)(x) =
        \tau f(\tau^{-1} x)$ if $\tau > 0$ and $\delta_{\{0\}}(x)$ otherwise. Its convex conjugate amounts to $(\tau \star f)^* = \tau f^*.$ Denote by $(f \mathbin{\square} \phi)(y):=\inf_{x\in X} f(x) + \phi(y-x)$ the infimal convolution of $f,\phi \in \Gamma_0(X)$. We say that $f \mathbin{\square} \phi$ is exact at $y$ for $x$ if $(f \mathbin{\square} \phi)(y)=f(x)+\phi(y-x)$. Let $C \subset X$ be a closed convex nonempty set. Then we denote by $\proj(y, C):=\argmin_{x \in C}\|x-y\|$ the Euclidean projection of $x$ onto $C$ and $\dist(y, C):=\inf_{x \in C}\|x-y\|$ is the Euclidean distance of $x$ to $C$. By $N_C(y):=\{v \in X : \langle x-y,v \rangle \leq 0,\;  \forall x \in C\}$ we denote the normal cone of $C$ at $y \in C$ with the convention $N_C(y)=\emptyset$ if $y \notin C$. $T$ is a set-valued mapping from $X$ to $X$, written $T: X \rightrightarrows X$ if $T$ is a mapping from $X$ to the power set of $X$, i.e., for all $x \in X$ we have $T(x) \subseteq X$. The graph of $T$ is given as $\gph T = \{ (x,x^*) \in X \times X : x^* \in T(x)\}$. Such a mapping is monotone if for all $(x,x^*),(y,y^*) \in \gph T$ we have $\langle x-y, x^* - y^*\rangle \geq 0$, and maximal monotone if its graph is not strictly contained in the graph of another monotone mapping. The domain of $T$ is $\dom T= \{x \in X : T(x) \neq \emptyset \}$ and the range of $T$ is $\ran T= T(X) = \bigcup_{x \in X} T(x)$. There always exists an inverse of $T$ which is the set-valued mapping $T^{-1}: X \rightrightarrows X$, defined via $T^{-1}(x^*):=\{x \in X : x^*\in T(x)\}$. We have the relation $\dom T = \ran T^{-1}$. We define the set of solutions or zeros of $T$ as $\zer T := T^{-1}(0)= \{ x \in X : 0 \in T(x)\}$. We denote by $B_\varepsilon(x)\subset X$ the closed $\varepsilon$-ball around $x$, where we write $B_\varepsilon$ if $x=0$.
Let $Q\subset X$ be open. We denote by $\mathcal{C}^k(Q)$ the class of functions $f:Q\to\bR$ that are $k$ times continuously differentiable on $Q$.
For $f\in \mathcal{C}^k(Q)$ and $x\in Q$, the $k$-th differential of $f$ at $x$ is denoted by
$D^k f(x):(X)^k\to \bR$,
and is a symmetric $k$-linear form.
For directions $u_1,\ldots,u_k\in\bR^n$, we write $D^k f(x)[u_1,\ldots,u_k]$ for its evaluation along these directions.
In particular, $\nabla f(x)=D f(x)$ and $\nabla^2 f(x)=D^2 f(x)$ denote the gradient and Hessian of $f$ at $x$, respectively.
\subsection{Outline}

The remainder of the paper is organized as follows.
In \cref{sec:problem} we recall the convex convex-composite problem class, its saddle-point reformulation, and the associated monotone KKT operator.
In \cref{sec:ppa} we review Legendre functions and Bregman distances and refine the existing convergence analysis of the inexact Bregman proximal point algorithm with relative errors \cite{solodov2000inexact} by means of a local convergence rate and an ergodic convergence result under more relaxed assumptions on the distance generating function. In \cref{sec:path_following_balm} we study a special case of the algorithm, the \emph{Bregman proximal augmented Lagrangian method}, providing a refined analysis and a convergence result of Newton's method for the solution of the subproblem under generalized self-concordance assumptions.
\Cref{sec:conclusion} concludes the paper.

\section{Convex convex-composite problems and duality} \label{sec:problem}
In this paper we attempt to solve the convex convex-composite optimization problem which takes the form
\begin{align} \label{eq:problem}
\inf_{x \in \bR^n} \left\{\varphi(x) \equiv f(x) + g(\mathcal{A}(x))\right\},
\end{align}
for convex, extended real-valued and proper lsc functions $f \in \Gamma_0(\bR^n)$ and $g \in \Gamma_0(\bR^m)$ and an affine linear mapping $\mathcal{A}:\bR^n \to \bR^m$ with $\mathcal{A}(x)=Ax-b$ for $A\in \bR^{m \times n}$ and $b\in \bR^m$.
We assume that $f$ is smooth on the interior of its domain.

The problem class is flexible and subsumes the following important standard examples encountered in convex optimization:
\begin{example}[quadratic programming] \label{ex:qp}
A quadratic program takes the form
\begin{align}
\min_{x \in Q} \tfrac{1}{2}\langle x, W x\rangle + \langle c, x\rangle \quad \text{subject to}\quad Ax = b,
\end{align}
for $W$ being positive semi-definite and $Q=\{x \in \bR^n \mid l_i \leq x_i \leq u_i\}$, for lower bounds $l_i \in \bR \cup\{-\infty\}$ and upper bounds $u_i \in \bR \cup\{+\infty\}$. This can be represented as \cref{eq:problem} via the choices $f(x)=\frac{1}{2}\langle x, W x\rangle + \langle c, x\rangle + \delta_Q(x)$ and $g=\delta_{\{0\}}$. 
\end{example}

\begin{example}[optimization with $\vecmax$] \label{ex:vecmax}
    Let $g:=\vecmax$ be the pointwise max function defined as $\vecmax(u):=\max\{u_1, u_2, \ldots, u_m\}$. This problem class is relevant in machine learning, robust optimization or games.
\end{example}

\begin{example}[optimization with $\|\cdot\|_1$] \label{ex:l1}
    Let $g:=\|\cdot\|_1$. This problem class subsumes problems in image and signal processing such as total variation minimization.
\end{example}

\begin{example}[semi-definite programming] \label{ex:sdp}
A semi-definite program in \emph{linear matrix inequality} (LMI) form reads
\begin{align}
\min_{x \in \bR^n} \langle x, c \rangle \quad \text{subject to}\quad \mathcal{A}(x) \preceq 0,
\end{align}
for $\mathcal{A}(x) = A_0 + A_1 x_1 + \ldots + A_n x_n$.
\end{example}
All examples fit within the general problem class \cref{eq:problem} via appropriate choices of $g,f$ and $\mathcal{A}$. Replacing $g$ with its biconjugate the problem \cref{eq:problem} admits a saddle-point version given as
\begin{align} \label{eq:saddle}
\inf_{x \in \bR^n} \sup_{y \in \bR^m} L(x,y),
\end{align}
where we call 
\begin{align} \label{eq:lagrangian}
L(x,y)=f(x) + \langle \mathcal{A}(x), y \rangle - g^*(y),
\end{align}
the convex-concave Lagrangian of the problem.

Throughout we assume that $L$ has a saddle-point meaning that primal and dual solutions exist and strong duality holds. It is important to note that for some problems such as \cref{ex:sdp} this is not automatic. In such cases one has to assume the following constraint qualification \cite[Example 11.41]{RoWe98}:
\begin{subequations}
\begin{align}
b &\in \intr (A \dom f - \dom g) \\
0 &\in \intr (A^\top \dom g^* + \dom f^*),
\end{align}
\end{subequations}
to guarantee existence of primal and dual solutions.
In our case, saddle-points
\begin{align}
\argminimax_{x\in \bR^n, y\in \bR^m} L(x,y),
\end{align}
can be characterized as zeros $T^{-1}(0) =:\zer T$ of the maximal monotone KKT operator:
\begin{align}
T(x,y) = \partial_x L(x,y) \times \partial_y (-L)(x,y) = (\partial f(x) + A^\top y )\times (\partial g^*(y) + b - Ax).
\end{align}

\section{Bregman proximal point method} \label{sec:ppa}
\subsection{Legendre functions and Bregman distances}
Akin to Lagrange--Newton methods and in particular interior point methods our approach attempts to find a zero of $T$ by solving a sequence of regularized inclusions using Newton's method.
This regularized inclusion is called \emph{Bregman proximal mapping} and the underlying algorithm is the Bregman proximal point algorithm
\cite{Eckstein93,eckstein1998approximate,burachik1998generalized,solodov2000inexact}. To account for inexactness in the solution of the regularized inclusion we consider a variant with relative errors \cite{solodov2000inexact}.
In contrast to the classical Euclidean proximal point algorithm, which relies on the Euclidean geometry, the Bregman proximal point algorithm utilizes the gradient of a Legendre function to measure proximity.

In this section we collect and extend some key results from \cite{Roc70,bauschke1997legendre,solodov2000inexact} to facilitate our analysis.

We begin by recalling the notion of a Legendre function \cite[Section~26]{Roc70}.

\begin{definition}[Legendre function] \label{def:legendre}
The function $\Phi \in \Gamma_0(\mathbb{E})$ is
\begin{defenum}
\item \emph{essentially smooth}, if $\intr(\dom \Phi) \neq \emptyset$ and $\Phi$ is differentiable on $\intr(\dom \Phi)$ such that $\|\nabla\Phi(z^\nu)\|\to \infty$, whenever $\intr(\dom \Phi) \ni z^\nu \to z \in \bdry \dom \Phi$, and
\item \emph{essentially strictly convex}, if $\Phi$ is strictly convex on every convex subset of $\dom \partial \Phi$, and
\item \emph{Legendre}, if $\Phi$ is both essentially smooth and essentially strictly convex.
\end{defenum}
\end{definition}

We list some basic properties of Legendre functions:
\begin{lemma} \label{thm:legendre_props}
Let $\Phi \in \Gamma_0(\mathbb{E})$ be Legendre. Then $\Phi$ has the following properties:
\begin{lemenum}
\item $\dom \partial \Phi = \intr (\dom\Phi)$, {\normalfont\rmfamily\cite[Theorem 26.1]{Roc70}}. \label{thm:legendre_props:dom_grad_dom_phi}
\item $\Phi^*$ is Legendre, {\normalfont\rmfamily\cite[Theorem 26.3]{Roc70}}. \label{thm:legendre_props:conjugate_legendre}
\item $\nabla \Phi :\intr (\dom\Phi) \to \intr (\dom\Phi^*)$ is a homeomorphism between $\intr (\dom\Phi)$ and $\intr (\dom\Phi^*)$, i.e., $\nabla \Phi$ is bijective with inverse $\nabla \Phi^*:\intr (\dom\Phi^*) \to \intr (\dom\Phi)$ and $\nabla \Phi$ and $\nabla \Phi^*$ are both continuous on $\intr (\dom\Phi)$ resp. $ \intr(\dom\Phi^*)$, {\normalfont\rmfamily\cite[Theorem 26.5]{Roc70}}. \label{thm:legendre_props:homeomorphism}
\item $\Phi$ is \emph{super-coercive} if and only if $\dom \Phi^*= \mathbb{E}$, {\normalfont\rmfamily\cite[Proposition 2.16]{bauschke1997legendre}}. \label{thm:legendre_props:super_coercive}
\end{lemenum}
\end{lemma}

We define the Bregman distance generated by the Legendre function $\Phi \in \Gamma_0(\mathbb{E})$ as:
\begin{align}
D_\Phi(z_1,z_2):=\begin{cases}
\Phi(z_1) - \Phi(z_2) - \langle \nabla \Phi(z_2), z_1-z_2 \rangle & \text{$z_1 \in \dom \Phi, z_2 \in \intr (\dom \Phi)$} \\
+\infty &\text{otherwise.}
\end{cases}
\end{align}
A Bregman distance generated by a Legendre function exhibits the following elementary properties
\begin{lemma} \label{thm:dual_bregman}
Let $\Phi \in \Gamma_0(\mathbb{E})$ be Legendre. Then the following hold
\begin{lemenum}
    \item $D_{\Phi}(z_1, \nabla \Phi^*(z_2^*)) = \Phi(z_1) + \Phi^*(z_2^*) -\langle z_1,z_2^*\rangle$, for any $z_1 \in \dom \Phi$ and $z_2^* \in \intr \dom \Phi^*$. \label{thm:dual_bregman:one_sided}
    \item $D_{\Phi}(z_1, z_2) \geq 0$ for all $z_1 \in \dom \Phi$ and $z_2 \in \intr \dom \Phi$ with equality if $z_1 = z_2$ \label{thm:dual_bregman:nonnegative}.
    \item $D_{\Phi^*}(z_1, z_2) = D_{\Phi}(\nabla \Phi^*(z_1), \nabla \Phi^*(z_2))$, for any $z_1,z_2 \in \intr \dom \Phi^*$. {\normalfont\rmfamily\cite[Theorem 3.7(v)]{bauschke1997legendre} }\label{thm:dual_bregman:flip}
    \item $D_\Phi(s,z) = D_\Phi(s,x) + \langle \nabla \Phi(x) - \nabla \Phi(z), s-y \rangle + D_\Phi(y,z) - D_\Phi(y,x)$, for any $x,z \in \intr(\dom \Phi)$ and $y,s \in \dom \Phi$. {\normalfont\rmfamily\cite[Corollary 2.6]{solodov2000inexact}} \label{thm:dual_bregman:fourpoint}
    \item Let $C \subset \intr \dom \Phi$ be compact. Then there exists $r>0$ and $\delta < \infty$ such that \label{thm:affine_bounds}
\begin{align}
D_\Phi(x,y) \geq r\|x\| - \delta,
\end{align}  
for every $x \in \dom \Phi$ and every $y \in C$.
\end{lemenum}
\end{lemma}
\begin{proof}
``\labelcref{thm:dual_bregman:one_sided,thm:dual_bregman:nonnegative}'': Let $z_2^* \in \intr \dom \Phi$. By \cref{thm:legendre_props:dom_grad_dom_phi,thm:legendre_props:homeomorphism} this means there exists $z_2 \in \intr \dom \Phi$ such that $z_2^* = \nabla \Phi(z_2)$. Then the claim follows by the Fenchel--Young inequality.

``\labelcref{thm:affine_bounds}'': Define $C^*= \nabla \Phi(C) \subset \intr \dom \Phi^*$ which is compact by continuity of $\nabla \Phi$ on $\intr \dom \Phi$. By compactness there exists $r>0$ such that $C^* + rB_1 \subset \intr \dom \Phi^*$. Define $\Theta:= \max \{ \Phi^*(z^*) : z^* \in C^* + rB_1\}$
and $\theta:= \min \{ \Phi^*(y^*) : y^* \in C^*\}$ which are finite by compactness and continuity of $\Phi^*$ on $\intr \dom \Phi^*$.

Take $y \in C$ and $x \in \dom \Phi$. 
Define $y^*=\nabla \Phi(y) \in C^*$. Assume for now that $x \neq 0$ so that $z^*:=y^* + r x/\|x\| \in C^*+r B_1$ and hence $\Phi^*(z^*) \leq \Theta$.

By Fenchel--Young we have that
\begin{align}
\Phi(x) + \Phi^*(z^*) \geq \langle x, z^*\rangle 
=\langle x, y^*\rangle + r\langle x, x/\|x\|\rangle 
=\langle x, y^*\rangle +r\|x\|,
\end{align}
which we rearrange to
\begin{align}
\Phi(x) -\langle x, y^*\rangle &\geq r\|x\| -  \Phi^*(z^*) \geq r\|x\| -  \Theta.
\end{align}
Adding $\Phi^*(y^*)$ to both sides of the inequality, we obtain via \cref{thm:dual_bregman:one_sided} that
\begin{align}
D_\Phi(x,y)=\Phi(x) +\Phi^*(y^*)-\langle x, y^*\rangle \geq r\|x\| -  \Theta+ \Phi^*(y^*) \geq  r\|x\| -(\Theta - \theta).
\end{align}
If $x=0$ we trivially have that
\[
D_\Phi(0, y)\geq 0 > -(\Theta - \theta). \qedhere
\]
\end{proof}

Next we state the definition of \emph{very strict convexity} \cite[Definition 2.8]{bauschke2000dykstras}:
\begin{definition}[very strict convexity]
Let $\Phi\in \Gamma_0(\mathbb{E})$ be twice continuously differentiable on $\intr\dom\Phi \neq \emptyset$. We say $\Phi$ is \emph{very strictly convex} if $\nabla^2 \Phi f(z) \succ 0$ for every $z \in \intr \dom \Phi$.
\end{definition}

\begin{lemma}[inverse Hessian identity] \label{thm:inverse_hessian}
Let $\Phi\in \Gamma_0(\mathbb{E})$ be Legendre and very strictly convex. Then $\Phi^*$ is Legendre and very strictly convex. Moreover, for any conjugate pair $z \in \intr(\dom\Phi)$ and $\nabla \Phi(z) \in \intr(\dom\Phi^*)$ the Hessian matrices $\nabla^2 \Phi(z)$ and $\nabla^2 \Phi^*(\nabla\Phi(z))$ are inverse to each other.
\end{lemma}
\begin{proof}
\cite[Lemma 1.18]{laude2021lower}
\end{proof}

Under \emph{very strict convexity} of Legendre function $\Phi$, Bregman distances and Euclidean distances are locally equivalent as shown in the following extension of \cite[Proposition 2.10]{bauschke2000dykstras}:
\begin{lemma}[local equivalence of Euclidean and Bregmanian geometry] \label{thm:very_strictly_convex}
    Let $\Phi\in \Gamma_0(\mathbb{E})$ be Legendre and twice continuously differentiable on $\intr\dom\Phi$. Let $C$ be a convex compact subset of $\intr \dom \Phi$. Then there exists $\Theta <\infty$ such that the following three equivalent conditions hold true:
    \begin{lemenum}
        \item $\|\nabla \Phi(z_1) - \nabla\Phi(z_2)\| \leq \Theta\|z_1 -z_2\|$ for all $z_1,z_2 \in C$ \label{thm:very_strictly_convex:lip}
        \item $D_\Phi(z_2, z_1) \leq \tfrac{\Theta}{2}\|z_1-z_2\|^2$ for all $z_1,z_2 \in C$ \label{thm:very_strictly_convex:descent}
        \item $D_\Phi(z_2, z_1) \geq \tfrac{1}{2\Theta}\|\nabla \Phi(z_2)-\nabla \Phi(z_1)\|^2$ for all $z_1,z_2 \in C$ \label{thm:very_strictly_convex:coco}
    \end{lemenum}
    If, in addition, $\Phi$ is very strictly convex there exists $\theta >0$ such that
    \begin{lemenum}[resume]
        \item $D_\Phi(z_1,z_2) \geq \tfrac{\theta}{2} \|z_1-z_2\|^2$ for all $z_1,z_2 \in C$ \label{thm:very_strictly_convex:strong}
    \end{lemenum}
\end{lemma}
\begin{proof}
``\labelcref{thm:very_strictly_convex:lip,thm:very_strictly_convex:strong}'': This follows by \cite[Proposition 2.10]{bauschke2000dykstras} and essential smoothness of $\Phi$.

``\labelcref{thm:very_strictly_convex:lip} $\Rightarrow$ \labelcref{thm:very_strictly_convex:descent}'':
    Since $C \subset \intr \dom \Phi$ is convex we have for $z_1,z_2 \in C$ invoking \cite[Theorem 2.1.5]{nesterov2018lectures} that
    \begin{align} \label{eq:descent_lemma_phi}
        D_\Phi(z_1, z_2) \leq \tfrac{\Theta}{2}\|z_1 -z_2\|^2.
    \end{align}

``\labelcref{thm:very_strictly_convex:descent} $\Rightarrow$ \labelcref{thm:very_strictly_convex:coco}'':
    Let $z_1,z_2 \in C$ and define $z_1^*:=\nabla \Phi(z_1)$ and $z_2^*:=\nabla \Phi(z_2)$. Hence $z_2=\nabla \Phi^*(z_2^*)$ and we have invoking \cref{thm:dual_bregman:one_sided}, $D_\Phi(z_1, z_2)=D_\Phi(z_1, \nabla \Phi^*(z_2^*))=\Phi(z_1) + \Phi^*(z_2^*) - \langle z_2^*, z_1\rangle$ and so we can rearrange \cref{eq:descent_lemma_phi} to
    \begin{align}
        -\Phi(z_1) \geq \Phi^*(z_2^*)-\langle z_2^*, z_1\rangle - \tfrac{\Theta}{2}\|z_1 -z_2\|^2.
    \end{align}
    We add $\langle z_1^*, z_1\rangle$ to both sides of the inequality and maximize both sides wrt. $z_1$:
    \begin{align}
        \Phi^*(z_1^*)=\sup_{z_1 \in \mathbb{E}}\langle z_1^*, z_1\rangle-\Phi(z_1) \geq \Phi^*(z_2^*) +\langle z_1^*-z_2^*,z_2 \rangle + \sup_{d \in \mathbb{E}} \langle z_1^*-z_2^*, d\rangle - \tfrac{\Theta}{2}\|d\|^2,
    \end{align}
    which yields via $z_2=\nabla \Phi^*(z_2^*)$ and the identity in \cref{thm:dual_bregman:flip}
    \begin{align}
        D_\Phi(z_2, z_1) = D_{\Phi^*}(z_1^*,z_2^*) \geq \tfrac{1}{2\Theta}\|\nabla \Phi(z_2)-\nabla \Phi(z_1)\|^2.
    \end{align}

``\labelcref{thm:very_strictly_convex:coco} $\Rightarrow$ \labelcref{thm:very_strictly_convex:lip}'':
    Let $z_1 \neq z_2 \in C$. Summing $D_\Phi(z_2, z_1) \geq \tfrac{1}{2\Theta}\|\nabla \Phi(z_2)-\nabla \Phi(z_1)\|^2$ with $z_1,z_2$ interchanged we obtain via Cauchy--Schwarz
    \begin{align}
        \|z_1 -z_2\|\|\nabla\Phi(z_1)-\nabla \Phi(z_2)\| \geq \tfrac{1}{2\Theta}\|\nabla \Phi(z_2)-\nabla \Phi(z_1)\|^2.
    \end{align}
    Dividing by $\|\nabla\Phi(z_1)-\nabla \Phi(z_2)\| \neq 0$ yields the claimed inequality.
\end{proof}
We conclude this section by providing examples for very strictly convex Legendre functions:
\begin{example}[Energy] \label{ex:energy}
 In the context of Bregman distances, $\Phi(z)=\frac{1}{2}\|z\|^2$ is called the energy. The resulting Bregman distance $D_\Phi(z_1,z_2)=\frac{1}{2}\|z_1 - z_2\|^2$ is the squared Euclidean distance.
\end{example}
\begin{example}[von Neumann entropy] \label{ex:neumann}
Let $\Phi(z):=\sum_{i=0}^m \varphi(z_i)$ with $\varphi(t)=t \ln(t)-t$ and $\dom \varphi=[0,+\infty)$ such that $0 \ln(0):=0$. Then $\Phi$ is Legendre and very strictly convex and $D_\Phi$ is the so-called  Kullback--Leibler divergence.
\end{example}
\begin{example}[Burg's entropy] \label{ex:burg}
    Let $\Phi(z):=\sum_{i=0}^m \varphi(z_i)$ with $\varphi(t)=-\ln(t)$ and $\dom \varphi=(0,+\infty)$. Then $\Phi$ is Legendre and very strictly convex and $D_\Phi$ is the so-called Itakura--Saito divergence.
\end{example}
The following novel example is based on Spence's function. To our knowledge, Spence's function has not been linked to Bregman distances and Legendre functions in existing literature except \cite[Example 4.15]{laude2025anisotropic} where a similar construction is used to derive a certain conic decomposition of quadratics.
\begin{example}[Spence's entropy] \label{ex:spence}
Define the separable function $\Phi(z)=\sum_{i=1}^m \varphi(z_i)$, herein called Spence's entropy, with
\[
\varphi(t):=\int_{0}^t \ln(\exp(\tau)-1)\,d\tau\qquad \text{with}\qquad \dom\varphi=[0,\infty),
\]
and thus
\[
\varphi'(t)=\ln(\exp(t)-1)\qquad \text{and} \qquad(\varphi')^{-1}(t^*)=\ln(1+\exp(t^*)).
\]
As a consequence $\Phi\in\Gamma_0(\mathbb{R}^m)$ is Legendre and very strictly convex.
Its convex conjugate $\Phi^*(z^*)=\sum_{i=1}^m \varphi^*(x_i^*)$ satisfies
\[
\varphi^*(t^*)
=
\int_0^{t^*}\ln(1+\exp(\tau))\,d\tau+\tfrac{\pi^2}{12}.
\]
It can be readily checked that $\Phi^*$ has $1$-Lipschitz continuous gradient and therefore $\Phi$ is $1$-strongly convex. We call the resulting Bregman distance $D_\Phi$ Spence's divergence.
\end{example}
 The advantage of using Spence's entropy over the classical von Neumann entropy becomes apparent in \cref{ex:softplus}: While both choices yield an augmented Lagrangian that is quasi self-concordant, the penalty induced by Spence’s entropy is additionally Lipschitz smooth and provides a smooth approximation of the classical max-quadratic penalty used for inequality constraints.
\begin{remark}
The integrals defining $\varphi$ and $\varphi^*$ admit closed-form
expressions in terms of the dilogarithm (Spence’s function)\footnote{Numerical implementations of $\Li_2$ are available in \textsc{python}: \url{https://docs.scipy.org/doc/scipy/reference/generated/scipy.special.spence.html}. However, these expressions are never evaluated in the algorithm:
all updates and globalization procedures rely solely on gradient
information.}
\[
\Li_2(t)=\sum_{k=1}^\infty \tfrac{t^k}{k^2}
= -\int_0^t \tfrac{\ln(1-\tau)}{\tau}\,d\tau,
\]
to which \cref{ex:spence} owes its name.
In particular,
\[
\varphi(t)=\tfrac{t^2}{2}+\Li_2(\exp(-t))-\tfrac{\pi^2}{6}
\qquad \text{and} \qquad
\varphi^*(t^*)=-\Li_2(-\exp(t^*)).
\]
\end{remark}
\subsection{Definition and well-definedness}
The Bregman proximal augmented Lagrangian method is rooted in the Bregman proximal point method \cite{Eckstein93,eckstein1998approximate,burachik1998generalized,solodov2000inexact} for computing a zero
\begin{align}
z^\star \in \zer T:=T^{-1}(0),
\end{align}
of the maximally monotone saddle-point operator $T:\mathbb{E} \rightrightarrows \mathbb{E}$. 
We assume that $\intr (\dom \Phi) \cap \dom T \neq \emptyset$.
Given a sequence of positive step-sizes $\{\sigma_k\}_{k=0}^\infty$ that are bounded away from zero $\sigma_k > \sigma > 0$ for all $k$ the Bregman proximal point algorithm takes the form:
\begin{align} \label{eq:exact_bregman_ppa}
z^{k+1} \in (\nabla \Phi + \sigma_k T)^{-1}(\nabla \Phi(z^k)).
\end{align}
Given $z^0 \in \intr (\dom \Phi)$ the iteration is well-defined in the sense that $z^{k+1} \in \intr (\dom \Phi)$ exists if $\Phi$ is Legendre and super-coercive \cite[Proposition 3]{burachik1998generalized}.

To account for errors we consider an inexact variant of the Bregman proximal point method due to \cite{solodov2000inexact} where $z^k$ is updated by computing a triplet $(z^{k+1}, p^k, w^k)$ that satisfies
\begin{subequations} \label{eq:solodov_bregman_ppa}
\begin{align}
w^{k} &\in T(p^{k}) \label{eq:solodov_1}\\
z^{k+1} &= \nabla \Phi^*(\nabla \Phi(z^k) - \sigma_k w^{k}) \in \intr (\dom \Phi), \label{eq:extra_gradient}
\end{align}
\end{subequations}
where $p^k$ and $z^{k+1}$ are ``near'' to each other as controlled using the following error criterion:
\begin{align}\label{eq:error_bound_solodov}
D_\Phi(p^{k}, z^{k+1}) &\leq \rho_k D_\Phi(p^{k}, z^k).
\end{align}
For well-definedness of the algorithm note that \cref{eq:exact_bregman_ppa} is equivalent to \cref{eq:solodov_bregman_ppa} if $p^k = z^{k+1}$ in which case the error criterion \cref{eq:error_bound_solodov} is satisfied for any $\rho_k\geq 0$.
Thus the existence of the iterates in \cref{eq:solodov_bregman_ppa} is implied by the existence of the exact Bregman proximal mapping \cref{eq:exact_bregman_ppa}. In fact, thanks to the Solodov--Svaiter stopping criterion~\cref{eq:error_bound_solodov}, the set of admissable next iterates $z^{k+1}$ is larger than for the exact version.

\subsection{Global convergence} \label{sec:global_ppa}
In this subsection we prove a generalization of \cite[Theorem 3.2(1)]{solodov2000inexact} dropping the requirement that $\Phi$ is a Bregman function.
In particular this allows us to analyze the algorithm when $\Phi$ involves Burg's entropy (\cref{ex:burg}) which is Legendre but not Bregman.
The following result is a slight refinement of \cite[Lemma 4.1]{solodov2000inexact}. For completeness a proof is provided.
\begin{lemma}[Fej\'er monotonicity] \label{thm:fejer_prox_solodov}
Let $z^\star \in \zer T \cap \dom \Phi$ and assume that $\Phi$ is Legendre. Let $0\leq \rho_k\leq \rho < 1$. Then we have that
\begin{align*}
D_\Phi(z^\star,z^{k+1}) &\leq D_\Phi(z^\star,z^k) - \sigma_k \langle w^{k}, p^{k}-z^\star \rangle -(1-\rho_k)D_\Phi(p^{k},z^k) \\
&\leq D_\Phi(z^\star,z^k) -(1-\rho)D_\Phi(p^{k},z^k).
\end{align*}
\end{lemma}
\begin{proof}
By assumption $z^\star \in \dom T \subseteq \dom \Phi$. Since $\Phi$ is Legendre we have by \eqref{eq:extra_gradient} that $z^{k+1} \in \intr(\dom \Phi)$.
We have thanks to \cref{thm:dual_bregman:fourpoint} for $s=z^\star$, $z=z^{k+1}$, $x = z^k$ and $y=p^{k}$:
\begin{align*}
D_\Phi(z^\star,z^{k+1}) &= D_\Phi(z^\star,z^k) + \langle \nabla \Phi(z^k) - \nabla \Phi(z^{k+1}), z^\star-p^{k} \rangle + D_\Phi(p^{k},z^{k+1}) - D_\Phi(p^{k},z^k) \\
&=D_\Phi(z^\star,z^k) -\sigma_k\langle w^{k}, p^{k}-z^\star \rangle + D_\Phi(p^{k},z^{k+1}) - D_\Phi(p^{k},z^k) \\ 
&\leq D_\Phi(z^\star,z^k)-\sigma_k \langle w^{k}-0, p^{k}-z^\star \rangle -(1-\rho_k)D_\Phi(p^{k},z^k) \\
&\leq D_\Phi(z^\star,z^k) -(1-\rho)D_\Phi(p^{k},z^k)
\end{align*}
where the second equality follows from the extra-gradient step \cref{eq:extra_gradient} in the definition of the algorithm and the last inequalities follow from the error bound \cref{eq:error_bound_solodov} and the monotonicity of $T$ respectively.
\end{proof}
The following result is a slight refinement of \cite[Corollary 4.3]{solodov2000inexact}. A proof is provided for completeness.
\begin{lemma}
\label{thm:fejer_prox_solodov_implications}
Let $z^\star \in \zer T \cap \dom \Phi$ and assume that $\Phi\in \Gamma_0(\mathbb{E})$ is Legendre. Let $0\leq \rho_k \leq \rho < 1$. Then the following properties hold true:
\begin{lemenum}
\item \label{thm:fejer_prox_solodov_implications:decr} $\{D_\Phi(z^\star,z^{k})\}_{k=0}^\infty$ is monotonically decreasing and convergent.
\item \label{thm:fejer_prox_solodov_implications:sum} $\{D_\Phi(p^{k},z^{k+1})\}_{k=0}^\infty$ and $\{D_\Phi(p^{k},z^{k})\}_{k=0}^\infty$ are summable and in particular converge to $0$.
\item \label{thm:fejer_prox_solodov_implications:sum_inner} $\sum_{k=0}^\infty \langle w^{k}, p^{k}-z^\star \rangle < \infty$ and in particular $\langle w^{k}, p^{k}-z^\star \rangle \to 0$.
\end{lemenum}
If, in addition, $D_\Phi(z^\star, \cdot)$ is coercive the following item can be added to the list
\begin{lemenum}[resume]
\item \label{thm:fejer_prox_solodov_implications:bounded} $\{z^{k}\}_{k=0}^\infty$ 
is bounded.
\end{lemenum}
\end{lemma}
\begin{proof}
``\labelcref{thm:fejer_prox_solodov_implications:decr}'': Thanks to the second inequality from \cref{thm:fejer_prox_solodov} we obtain that $\{D_\Phi(z^\star, z^k)\}_{k=0}^\infty$ is monotoncially decreasing. By nonnegativity of the Bregman distance it is also convergent. 

``\labelcref{thm:fejer_prox_solodov_implications:sum}'': Summing the second inequality from \cref{thm:fejer_prox_solodov} we obtain that
\begin{align} \label{eq:sum_lemma}
0 \leq (1-\rho)\sum_{k=0}^K D_\Phi(p^{k},z^{k}) \leq D_\Phi(z^\star, z^0) - D_\Phi(z^\star, z^{K+1}) \leq D_\Phi(z^\star, z^0),
\end{align}
and thus $\{D_\Phi(p^{k},z^k)\}_{k=0}^\infty$ is summable implying that $D_\Phi(p^{k},z^k) \to 0$. In light of \cref{eq:error_bound_solodov} $D_\Phi(p^{k},z^{k+1}) \leq \rho D_\Phi(p^{k},z^k)$ and thus we also have that $\{D_\Phi(p^{k},z^{k+1})\}_{k=0}^\infty$ is summable and $D_\Phi(p^{k},z^{k+1}) \to 0$.

``\labelcref{thm:fejer_prox_solodov_implications:sum_inner}'': Summing the first inequality from \cref{thm:fejer_prox_solodov} we obtain since $0< \sigma \leq \sigma_k$ that
\begin{align*}
\sigma \sum_{k=0}^K \langle w^{k}, p^{k}-z^\star \rangle &\leq \sum_{k=0}^K \sigma_k \langle w^{k}, p^{k}-z^\star \rangle \\
&\leq D_\Phi(z^\star,z^0)-D_\Phi(z^\star,z^{K+1}) \leq D_\Phi(z^\star,z^0).
\end{align*}
Dividing by $\sigma$ and passing $K \to \infty$ we obtain that $\sum_{k=0}^\infty \langle w^{k}, p^{k}-z^\star \rangle < \infty$. Since $0 \in T(z^\star)$ and $w^{k} \in T(p^{k})$ we obtain via monotonicity of $T$ that $0 \leq \langle w^{k}-0, p^{k}-z^\star \rangle$ and thus we have that $\langle w^{k} -0, p^{k}-z^\star \rangle \to 0$.

``\labelcref{thm:fejer_prox_solodov_implications:bounded}'': By assumption, $D_\Phi(z^\star,\cdot)$ is coercive for some $z^\star \in \zer T \cap \dom \Phi$. Since $\{D_\Phi(z^\star,z^{k})\}_{k=0}^\infty$ is monotonically decreasing this means that $\{z^k\}_{k=0}^\infty$ is bounded.
\end{proof}

The following result is a generalization of \cite[Proposition 4.4]{solodov2000inexact}. It establishes convergence of the Bregman proximal point algorithm with relative errors only assuming $\zer T \cap \intr (\dom \Phi) \neq \emptyset$ and $D_\Phi(z^\star, \cdot)$ coercive for some $z^\star \in \zer T \cap \intr (\dom \Phi) $.
\begin{theorem} \label{thm:convergence_interior}
Let $\Phi\in \Gamma_0(\mathbb{E})$ is Legendre. Assume that $D_\Phi(z^\star, \cdot)$ is coercive for some $z^\star \in \zer T \cap \intr (\dom \Phi) \neq \emptyset$ which happens to be the case if $\dom \Phi^*$ is open. Let $0\leq \rho_k\leq \rho < 1$. Then $z^k \to z^\infty \in \zer T \cap \intr (\dom \Phi)$, $w^k \to 0$ and $p^k \to z^\infty$.
\end{theorem}
\begin{proof}
In light of \cite[Theorem 3.7(vi)]{bauschke1997legendre}, $D_\Phi(z, \cdot)$ is coercive for every $z \in \intr \dom \Phi$ if $\dom \Phi^*$ is open.

In light of \cref{thm:fejer_prox_solodov_implications:bounded}, $\{z^k\}_{k=0}^\infty\subset \intr (\dom \Phi)$ is bounded and as such it has an accumulation point $z^\infty \in \cl \dom \Phi$ and $z^{k_j} \to z^\infty$ for some subsequence $\{z^{k_j}\}_{j=0}^\infty$. Take any $z^\star \in \zer T \cap \intr (\dom \Phi)$. Suppose that $z^\infty \in \bdry \dom(\Phi)$. Since $\Phi$ is Legendre and thus in particular essentially smooth in light of \cite[Theorem 3.8(i)]{bauschke1997legendre} we have that
\begin{align}
    \lim_{j \to \infty} D_\Phi(z^\star, z^{k_j}) \to +\infty.
\end{align}
By \cref{thm:fejer_prox_solodov_implications:decr} we have $D_\Phi(z^\star, z^{k_j})$ is convergent, a contradiction. Thus $z^\infty \in \intr (\dom \Phi)$. Similarly we can show that $z^{k_j+1} \to \bar z\in \intr \dom \Phi$ by going to another subsequence if necessary.

Next we show that $\bar z = z^\infty$. Since $z^{k_j} \to z^\infty \in \intr\dom \Phi$, there exists a compact set $C \subset \intr (\dom \Phi)$ such that $z^{k_j} \in C$ for $j$ sufficiently large. By \cref{thm:fejer_prox_solodov_implications:sum} we have that $\lim_{j \to \infty} D_\Phi(p^{k_j}, z^{k_j}) \to 0$. In light of \cref{thm:affine_bounds}, there exist $\delta, r$ finite, such that $D_\Phi(p^{k_j}, z^{k_j}) +\delta \geq r\|p^{k_j}\|$ implying that $\{p^{k_j}\}_{j=0}^\infty$ is bounded as well. By going to another subsequence if necessary $p^{k_j} \to p^\infty$. Since $\nabla \Phi$ is continuous relative to $\intr \dom \Phi$ and $z^\infty\in \intr \dom \Phi$, we have that $\nabla \Phi(z^{k_j}) \to \nabla \Phi(z^\infty)$. By \cref{thm:dual_bregman:one_sided} 
$$
D_\Phi(p^{k_j}, z^{k_j}) = \Phi(p^{k_j}) + \Phi^*(\nabla \Phi(z^{k_j})) - \langle \nabla \Phi(z^{k_j}), p^{k_j} \rangle,
$$
and so by lower semi-continuity:
\begin{align*}
\Phi(p^{\infty}) + \Phi^*(\nabla \Phi(z^{\infty})) - \langle \nabla \Phi(z^{\infty}), p^{\infty} \rangle &\leq \liminf_{j\to \infty} \Phi(p^{k_j}) + \Phi^*(\nabla \Phi(z^{k_j})) - \langle \nabla \Phi(z^{k_j}), p^{k_j} \rangle \\
&=\liminf_{j\to \infty} D_\Phi(p^{k_j}, z^{k_j}) = 0.
\end{align*}
Invoking \cref{thm:dual_bregman:one_sided} again we have that $D_\Phi(p^{\infty}, z^{\infty}) \leq 0$ and so $z^\infty=p^\infty$.
By \cref{thm:fejer_prox_solodov_implications:sum} we also have that $\lim_{j \to \infty} D_\Phi(p^{k_j}, z^{k_j+1}) \to 0$. By the same argument, since $z^{k_j+1}\to \bar z \in \intr \dom \Phi$ we also have that $\bar z=p^\infty=z^\infty$.
Since $\nabla \Phi$ is continuous on $\intr (\dom \Phi)\ni z^\infty$ and $z^k \in \intr (\dom \Phi)$ we have that
$$
\lim_{j \to \infty} \nabla \Phi(z^{k_j})-\nabla \Phi(z^{k_j+1}) = 0.
$$
By \cref{eq:extra_gradient} we have that $\lim_{j \to \infty} \sigma_{k_j} w^{k_j} = \lim_{j \to \infty} \nabla \Phi(z^{k_j})-\nabla \Phi(z^{k_j+1})=0$. And since $\sigma_k \geq \sigma > 0$ this implies that $\lim_{j \to \infty} w^{k_j} = 0$, where $w^{k_j} \in T(p^{k_j})$. Since $T$ is maximally monotone and $p^{k_j} \to z^\infty$ we have that $0 \in T(z^\infty)$. By continuity of $D_\Phi(z^\infty, \cdot)$ on $\intr \dom \Phi$ we have that $\lim_{j \to \infty} D_\Phi(z^\infty, z^{k_j}) = 0$. By \cref{thm:fejer_prox_solodov_implications:decr} $\{D_\Phi(z^\infty, z^{k})\}_{k=0}^\infty$ is convergent. And therefore 
$$
0=\lim_{j \to \infty} D_\Phi(z^\infty, z^{k_j}) = \lim_{k \to \infty} D_\Phi(z^\infty, z^{k}).
$$
Suppose that $z^k \not \to z^\infty$. This means that there exists a convergent subsequence indexed by $k_l$ so that $z^{k_l} \to \tilde z \neq z^\infty$ and by using the same argument as above we enforce via \cite[Theorem 3.8(i)]{bauschke1997legendre} that $\tilde z \in \intr \dom \Phi$.
Thus by continuity
$$
0=\lim_{k \to \infty} D_\Phi(z^\infty, z^{k}) = \lim_{l \to \infty} D_\Phi(z^\infty, z^{k_l}) =D_\Phi(z^\infty, \tilde z),
$$
a contradiction.
Thus it holds $z^k \to z^\infty$ and in particular $w^k \to 0$.
By \cref{thm:fejer_prox_solodov_implications:sum} we have that $D_\Phi(p^k, z^{k}) \to 0$. Suppose that $p^k \not \to z^\infty$. This means that there exists a convergent subsequence indexed by $k_n$ so that $p^{k_n} \to \bar p \neq z^\infty$. By continuity we have that
$$
0=\lim_{k \to \infty} D_\Phi(p^k, z^{k}) = \lim_{n \to \infty} D_\Phi(p^{k_n}, z^{k_n}) =D_\Phi(\bar p, z^\infty),
$$
a contradiction. Thus we also have $p^k \to z^\infty$.
\end{proof}

\subsection{Convergence rates} \label{sec:rates_ppa}
We assume that the following error bound holds true at a solution $z^\star \in \zer T$.
\begin{align} \label{eq:error_bound}
\exists \delta, \kappa > 0 : z \in T^{-1}(w), \|z - z^\star\| \leq \delta, \|w\| \leq \delta \Rightarrow \dist(z,\zer T) \leq \kappa \|w\|.
\end{align}
Note that the condition is implied by metric subregularity of $T$ at $z^\star$ for $w=0$.

The convergence rate is stated in terms of the Bregman distance to the convex set of solutions $\zer T$ which is introduced next:
\begin{definition}[Bregman projection and Bregman distance to a convex set]
Let $C$ be a closed convex set such that $C \cap \dom \Phi \neq \emptyset$. Then we define the Bregman projection
\begin{align}
\proj_\Phi(y, C) &= \argmin_{x \in C} D_\Phi(x, y),
\end{align}
and accordingly the Bregman distance to a convex set
\begin{align}
\dist_\Phi(y, C) &= \inf_{x \in C} D_\Phi(x, y).
\end{align}
\end{definition}

\begin{lemma} \label{thm:bregman_projection}
Let $C$ be a closed convex set with $C \cap \intr \dom \Phi \neq \emptyset$. Then the following are true:
\begin{lemenum}
\item $\proj_\Phi(y, C) \in \intr \dom \Phi \cap C$ is single-valued for every $y \in \intr \dom \Phi$ \label{thm:bregman_projection:intr}
\item $D_\Phi(x, \proj_\Phi(y, C)) \leq D_\Phi(x, y) - D_\Phi(y,\proj_\Phi(y, C))$ for every $x \in \dom \Phi \cap C$ and $y \in \intr \dom \Phi$  \label{thm:bregman_projection:nonexpansive}
\item Assume that $D_\Phi(x, \cdot)$ is coercive for some $x \in \dom \Phi \cap C$ then $\proj_\Phi(\cdot, C)$ is continuous on $\intr \dom \Phi$.  \label{thm:bregman_projection:continuity}
\end{lemenum}
\end{lemma}
\begin{proof}
``\labelcref{thm:bregman_projection:intr,thm:bregman_projection:nonexpansive}'': \cite[Proposition 3.16]{bauschke1997legendre}.

``\labelcref{thm:bregman_projection:continuity}'': Let $\{y^k\}_{k=0}^\infty \subset \intr \dom \Phi$ with $y^k \to y \in \intr \dom \Phi$. Let $x \in \dom \Phi \cap C$ such that $D_\Phi(x, \cdot)$ is coercive and observe that via \labelcref{thm:bregman_projection:nonexpansive} we have for $z^k:= \proj_\Phi(y^k, C)$
\begin{align}
D_\Phi(x, z^k) \leq D_\Phi(x, y^k) \leq \sup_{k} D_\Phi(x, y^k) < \infty,
\end{align}
where the last inequality holds since $D_\Phi(x, \cdot)$ is continuous relative to $\intr \dom \Phi$ and $y \in \intr \dom \Phi$. By coercivity of $D_\Phi(x, \cdot)$, $z^k$ must be bounded. Now consider a convergent subsequence $z^{k_j} \to \bar z$. Suppose that $\bar z \in \bdry \dom \Phi$. Since $\Phi$ is essentially smooth in light of \cite[Theorem 3.8(i)]{bauschke1997legendre} we have that
\begin{align}
    \lim_{j \to \infty} D_\Phi(x, z^{k_j}) \to +\infty,
\end{align}
hence $\bar z \in \intr \dom \Phi$.
In light of \cite[Proposition 3.16]{bauschke1997legendre} it holds for the Bregman projection $z^k=\proj_\Phi(y^k, C)$:
\begin{align}
\langle \nabla \Phi(z^{k_j}) - \nabla \Phi(y^{k_j}), z - z^{k_j} \rangle \geq 0,
\end{align}
for all $z \in C$. Fix $z \in C$ and pass to the limit as $j \to \infty$. Then we have via continuity of $\nabla \Phi$ on $\intr \dom \Phi$
\begin{align}
\langle \nabla \Phi(\bar z) - \nabla \Phi(y), z - \bar z \rangle \geq 0,
\end{align}
and thus again via \cite[Proposition 3.16]{bauschke1997legendre},  $\bar z = \proj_\Phi(y, C)$.
\end{proof}

\begin{lemma}
    Let $\Phi\in \Gamma_0(\mathbb{E})$ be Legendre and very strictly convex and suppose that $\zer T \cap \intr\dom \Phi \neq \emptyset$.
    Let $\{z^k\}_{k=0}^\infty$ be the sequence of iterates generated by Algorithm~\ref{eq:solodov_bregman_ppa} and let $z^\infty \in \zer T$ denote its limit point. Suppose that $D_\Phi(z^\star, \cdot)$ is coercive for some $z^\star \in \dom \Phi \cap \zer T$ and \cref{eq:error_bound} holds true at $z^\infty$. 
    Then we have
    \begin{align}
 \tfrac{1}{C_k}\dist_\Phi(z^{k+1}, \zer T) &\leq D_\Phi(p^k, z^k),
    \end{align}
    for $C_k:=\big(\sqrt{\rho_k} \tfrac{\sqrt{\Theta}}{\sqrt{\theta}}+\tfrac{\kappa}{\sigma_k}\Theta(1+\sqrt{\rho_k})\big)^2 \tfrac{\Theta}{\theta}$.
\end{lemma}
\begin{proof} 
    By \cref{thm:convergence_interior} it holds that $w^k \to 0$, $z^k \to z^\infty \in \zer T \cap \intr \dom \Phi$ and $p^k \to z^\infty$ as $k \to \infty$. Since \cref{eq:error_bound} holds at $z^\infty$ we have that for $k$ sufficiently large
    \begin{align}
    \dist(p^{k}, \zer T) \leq \kappa \|w^k\|.
    \end{align}
    Furthermore since $z^\infty \in \intr \dom \Phi$ there exists a closed ball $B_\varepsilon(z^\infty)\subseteq \intr\dom \Phi$ around $z^\infty$ such that for $k$ sufficiently large, $z^k,p^k \in B_\varepsilon(z^\infty)$. 
    Invoking \cref{thm:very_strictly_convex:coco} for $z_1=z^{k+1}$ and $z_2=p^k$ and applying the error criterion~\cref{eq:error_bound_solodov} we obtain
    \begin{align} \label{eq:additional_error_criterion}
        \tfrac{1}{2}\|\nabla \Phi(p^k)-\nabla \Phi(z^{k+1})\|^2 &\leq \Theta D_\Phi(p^k, z^{k+1}) \leq \Theta \rho_k D_\Phi(p^k, z^k) \leq \Theta^2\rho_k \tfrac{1}{2}\|p^k - z^k\|^2,
    \end{align}
    where the last inequality follows from \cref{thm:very_strictly_convex:descent}.
Via the triangle inequality and \cref{eq:additional_error_criterion} we have invoking \cref{thm:very_strictly_convex:lip} that
    \begin{align}
    \|w^k\| &= \tfrac{1}{\sigma_k}\|\nabla \Phi(z^k) - \nabla \Phi(z^{k+1})\| \notag\\
    &= \tfrac{1}{\sigma_k}\|\nabla \Phi(z^k)
    - \nabla \Phi(z^{k+1})\| \notag\\
 &\leq \tfrac{1}{\sigma_k}\|\nabla \Phi(z^k) - \nabla \Phi(p^k)\|+ \tfrac{1}{\sigma_k}\|\nabla \Phi(p^k) - \nabla \Phi(z^{k+1})\| \notag\\
 &\leq \tfrac{1}{\sigma_k}\|\nabla \Phi(z^k) - \nabla \Phi(p^k)\| + \tfrac{1}{\sigma_k}\sqrt{\rho_k} \Theta \|z^{k} - p^k\| \notag\\
 &\leq \tfrac{1}{\sigma_k}\Theta(1+\sqrt{\rho_k})\|z^k - p^k\|
    \end{align}
    and hence
    \begin{align} \label{eq:bound_euclidean_23}
       \dist(p^{k}, \zer T) \leq \tfrac{\kappa}{\sigma_k}\Theta(1+\sqrt{\rho_k})\|z^k - p^k\|.
    \end{align}
    Also note that via \cref{thm:very_strictly_convex:strong} the error criterion~\cref{eq:error_bound_solodov} yields
    \begin{align} \label{eq:bound_euclidean_2}
        \tfrac{1}{2}\|p^k-z^{k+1}\|^2 &\leq\tfrac{1}{\theta} D_\Phi(p^k, z^{k+1}) \leq \rho_k \tfrac{1}{\theta}D_\Phi(p^k, z^{k}) \leq \rho_k \tfrac{\Theta}{2\theta}\|p^k-z^{k}\|^2
    \end{align}
    By the triangle inequality using \cref{eq:bound_euclidean_23,eq:bound_euclidean_2} we can bound
    \begin{align}
    \dist(z^{k+1}, \zer T) &= \inf_{z \in \zer T} \|z -z^{k+1}\| \notag\\
    &\leq \inf_{z \in \zer T} \|z -p^k\| +\|p^k-z^{k+1}\| \notag\\
    &\leq \dist(p^{k}, \zer T) +\|p^k-z^{k+1}\| \notag\\
    &\leq \dist(p^{k}, \zer T) + \sqrt{\rho_k} \tfrac{\sqrt{\Theta}}{\sqrt{\theta}}\|p^k-z^{k}\| \notag \\
    &\leq \big(\sqrt{\rho_k} \tfrac{\sqrt{\Theta}}{\sqrt{\theta}}+\tfrac{\kappa}{\sigma_k}\Theta(1+\sqrt{\rho_k})\big)\|z^k - p^k\|
    \end{align}
    Via \cref{thm:very_strictly_convex:strong} we obtain after squaring the inequality
    \begin{align} \label{eq:bound_euclidean_bregman_1}
    \tfrac{1}{2}\dist^2(z^{k+1}, \zer T) &\leq \big(\sqrt{\rho_k} \tfrac{\sqrt{\Theta}}{\sqrt{\theta}}+\tfrac{\kappa}{\sigma_k}\Theta(1+\sqrt{\rho_k})\big)^2  \tfrac{1}{2}\|z^k - p^k\|^2 \notag\\
    &\leq \big(\sqrt{\rho_k} \tfrac{\sqrt{\Theta}}{\sqrt{\theta}}+\tfrac{\kappa}{\sigma_k}\Theta(1+\sqrt{\rho_k})\big)^2 \tfrac{1}{\theta}D_\Phi(p^k, z^k).
    \end{align}

    Denote by $\bar z^{k+1} := \proj_\Phi(z^{k+1}, \zer T)$.
    In light of \cref{thm:convergence_interior} $\intr \dom \Phi \ni z^{k+1} \to z^\infty \in \intr \dom \Phi$ and using the continuity of $\proj_\Phi(\cdot, \zer T)$ relative to $\intr \dom \Phi$, \cref{thm:bregman_projection:continuity}, we have that $\bar z^{k+1} = \proj_\Phi(z^{k+1}, \zer T) \to \proj_\Phi(z^\infty, \zer T)=z^\infty$ as $k\to \infty$.
    
    In particular, this means that $\bar z^{k+1} \in B_\varepsilon(z^\infty)$ for $k$ sufficiently large.
    Similarly note that the Euclidean projection $\proj(\cdot, \zer T)$ is nonexpansive. Hence for $\tilde{z}^{k+1}=\proj(z^{k+1}, \zer T)$ we have
\begin{align}
    \|z^\infty-\tilde{z}^{k+1}\| \leq \|z^\infty-z^{k+1}\|,
\end{align}
implying that for $k$ sufficiently large $\tilde{z}^{k+1}\in B_\varepsilon(z^\infty)$. 
    This implies via \cref{thm:very_strictly_convex:lip} that for $k$ sufficiently large,
    \begin{align} \label{eq:bound_euclidean_bregman_2}
    \tfrac{\Theta}{2}\dist^2(z^{k+1}, \zer T) &= \inf_{z \in \zer T} \tfrac{\Theta}{2}\|z^{k+1} - z\|^2 \notag \\
    &= \inf_{z \in \zer T\cap B_\varepsilon(z^\infty)} \tfrac{\Theta}{2}\|z^{k+1} - z\|^2 \notag \\
    &\geq \inf_{z \in \zer T\cap B_\varepsilon(z^\infty)} D_\Phi(z, z^{k+1}) = \dist_\Phi(z^{k+1}, \zer T).
    \end{align}
    Combining \cref{eq:bound_euclidean_bregman_1,eq:bound_euclidean_bregman_2} yields the desired result.
\end{proof}

\begin{theorem}
Let $\Phi\in \Gamma_0(\mathbb{E})$ be very strictly convex and suppose that $\zer T \cap \intr\dom \Phi \neq \emptyset$.
    Let $\{z^k\}_{k=0}^\infty$ be the sequence of iterates generated by Algorithm~\ref{eq:solodov_bregman_ppa} and let $z^\infty \in \zer T$ denote its limit point. Suppose that \cref{eq:error_bound} holds true at $z^\infty$. Then for $k$ sufficiently large we have that
    \begin{align*}
        \dist_\Phi(z^{k+1}, Z)\leq q_k \dist_\Phi(z^k, Z),
    \end{align*}
    for $q_k= \frac{1}{1+ \tfrac{1-\rho_k}{C_k}}$ with $0 < q_k < 1$ and $C_k:=\big(\sqrt{\rho_k} \tfrac{\sqrt{\Theta}}{\sqrt{\theta}}+\tfrac{\kappa}{\sigma_k}\Theta(1+\sqrt{\rho_k})\big)^2 \tfrac{\Theta}{\theta}$.
    In particular this means that $\dist_\Phi(z^k, Z)$ converges to $0$ superlinearly if $\sigma_k \to \infty$ and $\rho_k \to 0$.
\end{theorem}
\begin{proof}
Let $\bar z^k:=\proj_\Phi(z^k, \zer T)$. Then $D_\Phi(\bar z^k, z^{k}) = \dist_\Phi(z^k, Z)$ and thanks to the Fej\'er-bound \cref{thm:fejer_prox_solodov} we have
\begin{align}
    D_\Phi(\bar z^k, z^{k+1}) \leq \dist_\Phi(z^k, Z) - (1 -\rho_k)D_\Phi(p^k,z^k).
\end{align}
Thus we can bound
\begin{align}
   \dist_\Phi(z^{k+1}, Z) &= \inf_{z \in Z} D_\Phi(z, z^{k+1}) \notag\\
   &\leq D_\Phi(\bar z^k, z^{k+1}) \notag\\
   &\leq \dist_\Phi(z^k, Z)- (1 -\rho_k)D_\Phi(p^k,z^k) \notag\\
   &\leq  \dist_\Phi(z^k, Z) - \tfrac{1-\rho_k}{C_k} \dist_\Phi(z^{k+1}, Z).
\end{align}
Rearranging yields:
\begin{align}
  \dist_\Phi(z^{k+1}, Z) &\leq  \frac{1}{1+ \tfrac{1-\rho_k}{C_k}}\dist_\Phi(z^k, Z) \notag\\
  &= q_k\dist_\Phi(z^k, Z)
\end{align}
for $q_k=  \frac{1}{1+ \tfrac{1-\rho_k}{C_k}} \in(0,1)$. In particular we have $C_k \to 0$ as $\rho_k \to 0$ and $\sigma_k \to \infty$ implying that $q_k \to 0$. 
\end{proof}

\subsection{Ergodic convergence without domain interiority resriction} \label{sec:ergodic_ppa}
Notably, the constraint qualification $\intr \dom \Phi \cap \zer T \neq \emptyset$ is somewhat restrictive. However, to our knowledge, this cannot be avoided in a Fej\'er-monotonicity based analysis.

Hence, in this subsection, we prove an ergodic convergence result which allows us to circumvent the domain interiority requirement by considering the ergodic iterates for $0 \leq K$:
\begin{align} \label{eq:ergodic_four_point}
    \breve p^K := \frac{\sum_{k=0}^{K} \sigma_k p^{k}}{\sum_{k=0}^{K} \sigma_k}.
\end{align}

Next we prove the following auxiliary result:
\begin{lemma} \label{thm:ergodic_four_point}
Let $\Phi \in \Gamma_0(\mathbb{E})$ be Legendre. Let $z \in \dom \Phi$. Then it holds for the sequence of iterates $\{(p^k,z^k, w^k)\}_{k=0}^\infty$:
\begin{align}
\langle w^k,  p^k - z\rangle &\leq \tfrac{1}{\sigma_k}\big(D_\Phi(z,z^k) -  D_\Phi(z,z^{k+1}) - (1-\rho_k) D_\Phi(p^k,z^{k})\big).
\end{align}
\end{lemma}
\begin{proof}
Using the extra gradient step \cref{eq:extra_gradient} we obtain by rearranging
\begin{align} \label{eq:gap_lemma}
  \langle w^k,  p^k - z\rangle =\tfrac{1}{\sigma_k}\langle \nabla \Phi(z^k) -\nabla \Phi(z^{k+1}), p^k - z \rangle.
\end{align}
We have thanks to \cref{thm:dual_bregman:fourpoint} for $s=p^k$, $z=z^{k+1}$, $x = z^k$ and $y=z$:
\begin{align} \label{eq:four_point_lemma}
D_\Phi(p^k,z^{k+1}) = D_\Phi(p^k,z^k) + \langle \nabla \Phi(z^k) - \nabla \Phi(z^{k+1}), p^k-z \rangle + D_\Phi(z,z^{k+1}) - D_\Phi(z,z^k).
\end{align}
Combining \cref{eq:gap_lemma,eq:four_point_lemma} we obtain using the error criterion \cref{eq:error_bound_solodov}
\begin{align}
    \langle w^k,  p^k - z\rangle &=\tfrac{1}{\sigma_k}\big(D_\Phi(z,z^k) -  D_\Phi(z,z^{k+1}) + D_\Phi(p^k,z^{k+1}) -D_\Phi(p^k,z^k)\big) \notag \\
   &\leq \tfrac{1}{\sigma_k}\big(D_\Phi(z,z^k) -  D_\Phi(z,z^{k+1}) - (1-\rho_k) D_\Phi(p^k,z^{k})\big).
\end{align}
This completes the proof.
\end{proof}

Next we establish the following key inequality for the ergodic iterates $\breve p^K$:
\begin{lemma} \label{thm:ergodic_fixed}
Let $\Phi \in \Gamma_0(\mathbb{E})$ be Legendre. Let $z \in \dom T \cap \dom \Phi \neq \emptyset$. Then it holds for every $w \in T(z)$
\begin{align}
\langle w, \breve p^K - z\rangle &\leq \frac{D_\Phi(z,z^0)}{\sum_{k=0}^K \sigma_k}.
\end{align}
\end{lemma}
\begin{proof}
By monotonicity of $T$ we have
\begin{align}
    \langle w^k - w, p^k - z\rangle \geq 0.
\end{align}
Rearranging and invoking \cref{thm:ergodic_four_point} we obtain since $\dom T \subseteq \dom \Phi$:
\begin{align} \label{eq:gap_lemma_telescoping}
    \langle w,  p^k - z\rangle &\leq \langle w^k,  p^k - z\rangle \notag\\
   &\leq \tfrac{1}{\sigma_k}\big(D_\Phi(z,z^k) -  D_\Phi(z,z^{k+1}) - (1-\rho_k) D_\Phi(p^k,z^{k})\big).
\end{align}
We plug in the ergodic iterate $\breve p^K$ and obtain by telescoping since $\rho_k < 1$
\begin{align}
    \langle w, \breve p^K - z\rangle &= \frac{1}{\sum_{k=0}^K \sigma_k} \sum_{k=0}^K \sigma_k \langle w, p^k - z\rangle \\
    &\leq \frac{1}{\sum_{k=0}^K \sigma_k} \sum_{k=0}^K\big(D_\Phi(z,z^k) -  D_\Phi(z,z^{k+1}) - (1-\rho_k) D_\Phi(p^k,z^{k})\big) \\
    &\leq \frac{1}{\sum_{k=0}^K \sigma_k} \big(D_\Phi(z,z^0) -  D_\Phi(z,z^{K+1})\big)
\end{align}
This implies the claimed inequality.
\end{proof}

We define for radius $R$ the restricted gap function (in short gap) of the maximal monotone operator $T$
\begin{align} \label{eq:primal_dual_gap}
    \mathcal{G}_R(p, T):= \sup_{z \in B_R} \sup_{w \in T(z)} \langle w, p - z\rangle.
\end{align}
The restricted gap function is closely related to gap functions for variational inequalities \cite{chen1997gap} as well as the Fitzpatrick function \cite[Definition 20.51]{BaCo110}. It naturally serves as a measure of ``optimality'' as for $R > 0$ sufficiently large, it is convex proper, lsc and nonnegative while it attains the value $0$ if and only if $p \in \zer T$. All of these properties are summarized in the following lemma:
\begin{lemma} \label{thm:gap_optimality}
	Let $\zer T \neq \emptyset$.
    For any $R > 0$ sufficiently large the restricted gap function $\mathcal{G}_R(\cdot, T)$ is nonnegative, convex, proper and lsc. Furthermore for any $p \in \mathbb{E}$ and any $R> \|p\|$ sufficiently large $\mathcal{G}_R(p)=0$ if and only if $p\in \zer T$.
\end{lemma}
\begin{proof}
    Clearly, $\mathcal{G}_R(\cdot, T)$ is convex and lsc as a pointwise supremum over affine functions. Choose $R$ sufficiently large such that $\zer T \cap B_R \neq \emptyset$. Let $z^\star \in \zer T \cap B_R$ Then we have for any $p \in \mathbb{E}$, 
    \begin{align} \label{eq:gap_nonnegativity}
    \mathcal{G}_R(p, T) \geq \sup_{w \in T(z^\star)} \langle w, p - z^\star \rangle \geq \langle 0, p - z^\star \rangle =0.
    \end{align}
    Let $p \in \zer T$ and $R> \|p\|$.
    Since $T$ is maximal monotone in light of \cite[Definition 20.20]{BaCo110} this means that for all $(z,w) \in \gph T$
    \begin{align}
        \langle w, p-z\rangle \leq 0, 
    \end{align}
    implying that $\mathcal{G}_R(p, T) \leq 0$ and in combination with \cref{eq:gap_nonnegativity} $\mathcal{G}_R(p, T)=0$.  In particular, this implies that $\mathcal{G}_R(\cdot, T)$ is proper.

    Conversely, assume that $\mathcal{G}_R(p, T)=0$. By definition this means $\sup_{z \in B_R} \sup_{w \in T(z)} \langle w, p - z\rangle=0$. Since $z=p$ yields $\langle w, p - z\rangle=0$ this is equivalent to
    \begin{align} \label{eq:gap_ineq}
        \langle w, p - z\rangle\leq 0,
    \end{align}
    for all $z \in B_R$ and $w \in T(z)$. Consider the operator $T_R=T+N_{B_R}$ and take $(z,u)\in \gph T_R$. Hence we have that $u = w + n$ for $w \in T(z)$ and $n \in N_{B_R}(z)$ and $z \in B_R$.
    We have
    \begin{align} \label{eq:minty_tr}
        \langle u, p - z\rangle = \langle w, p - z\rangle + \langle n, p - z\rangle \leq 0 + 0,
    \end{align}
    where the last inequality follows from \cref{eq:gap_ineq} and the definition of the normal cone of $B_R$ at $z\in B_R$.
    By choosing $R$ sufficiently large we have that $\intr B_R \cap \dom T \neq \emptyset$ and thus via \cite[Theorem 1]{rockafellar1970maximality}, $T_R$ is maximal monotone. Thus \cref{eq:minty_tr} implies that $0 \in T_R(p)$.
    Since $\|p\| <R$ we have that $N_{B_R}(p)=\{0\}$ and thus $T_R(p)= T(p)$.
\end{proof}

\begin{proposition}
    Let $\Phi \in \Gamma_0(\mathbb{E})$ be Legendre. Furthermore, assume that $\dom T \subset \dom \Phi$.
    Then for every $R$ sufficiently large
   \begin{align}
    \mathcal{G}_R(\breve p^K, T) \leq  \frac{\sup_{z \in B_R \cap \dom T} D_\Phi(z,z^0)}{\sum_{k=0}^K \sigma_k}.
\end{align}
    This implies that $
    \mathcal{G}_R(\breve p^K, T) \to 0$
    as $K \to \infty$ and in particular every limit point $\breve p^\infty$ of the sequence $\{\breve p^K\}_{K=0}^\infty$ is a zero of $T$.
\end{proposition}
\begin{proof}
Choose $R$ sufficiently big. Then, in light of \cref{thm:gap_optimality}, $\mathcal{G}_R(\cdot, T)$ is proper convex, lsc and nonnegative.
Using the inequality from \cref{thm:ergodic_fixed}, taking the supremum over $z \in B_R \cap \dom T$ and $w \in T(z)$ on both sides of the inequality we have
\begin{align}
    0 &\leq \mathcal{G}_R(\breve p^K, T) \leq  \frac{\sup_{z \in B_R\cap \dom T} D_\Phi(z,z^0)}{\sum_{k=0}^K \sigma_k}.
\end{align}
Passing to the limit $K \to \infty$, since $\sigma_k > \sigma > 0$, this implies that $\mathcal{G}_R(\breve p^K, T) \to 0$.
In light of \cref{thm:fejer_prox_solodov_implications}, $\{p^k\}_{k=0}^\infty$ is bounded and hence the same is true for $\{\breve p^K\}_{K=0}^\infty$. 
Let $\breve p^\infty$ be a limit point of the sequence $\{\breve p^K\}_{K=0}^\infty$ and $\breve p^{K_j} \to \breve p^\infty$ for some subsequence $\{\breve p^{K_j}\}_{j=0}^\infty$.
Since $\mathcal{G}_R(\cdot, T)$ is lsc and nonnegative we have
\[
0 \leq \mathcal{G}_R(\breve p^\infty, T) \leq \liminf_{j \to \infty} \mathcal{G}_R(\breve p^{K_j}, T) = 0.
\]
Invoking \cref{thm:gap_optimality} this means $\breve p^\infty \in \zer T$.
\end{proof}

\section{Path-following Bregman proximal augmented Lagrangian method} \label{sec:path_following_balm}
\subsection{Definition and well-definedness}
As outlined in \cref{sec:problem} our approach attempts to find a solution to the saddle-point problem \cref{eq:saddle} by computing a zero of the monotone operator $T:=\partial_x L(x,y) \times \partial_y (-L)(x,y)$ given as the set-product of the subdifferentials of the Lagrangian \cref{eq:lagrangian}.
The underlying framework is the Bregman proximal point algorithm described in \cref{sec:ppa} which, in a nutshell, solves a sequence of regularized inclusions. 

Whenever $\Phi$ separates between decision and multiplier variables, the multiplier of this inclusion can be marginalized out so that the subproblem boils down to the minimization of the so-called Bregman augmented Lagrangian function plus a damping term.

To this end let $\psi \in \Gamma_0(\bR^n)$ and $\phi \in \Gamma_0(\bR^m)$ be Legendre and define the separable Legendre function $\Phi(z):=\psi(x)+\phi(y)$ for $z=(x,y)$. Then the Bregman distance $D_\Phi$ can be written as $D_\Phi(p, z)=D_\psi(s,x) + D_\phi(r, y)$. By separability the algorithm produces the iterates $w^k = (v^{k}, u^{k}) \in T(s^{k}, r^{k})$ for $p^k=(s^{k}, r^{k})$ in the product space. Furthermore, we assume that the dual update is computed exactly, i.e., $r^{k}=y^{k+1}$ and thus $D_{\phi}(r^{k}, y^{k+1})=0$.

Thanks to the separable structure of $\Phi$, the inclusion \cref{eq:solodov_bregman_ppa} for the $z$-update, separates into a primal and a dual inclusion
\begin{subequations} \label{eq:solodov_bregman_alm}
\begin{align}
    (v^{k}, u^{k}) &\in T(s^{k}, y^{k+1}) \label{eq:alm_solodov_1}\\
x^{k+1} &= \nabla \psi^*(\nabla \psi(x^k) - \sigma_k v^{k}) \label{eq:correction_solodov_primal} \\
y^{k+1} &= \nabla \phi^*(\nabla \phi(y^k) - \sigma_k u^{k}),\label{eq:correction_solodov_dual}
\end{align}
\end{subequations}
and the relative error criterion \cref{eq:error_bound_solodov} reads:
\begin{align}
D_\psi(s^{k},x^{k+1})&\leq \rho_k D_\psi(s^{k},x^k) + \rho_k D_\phi(y^{k+1}, y^k)\label{eq:error_bound_solodov_alm_plain},
\end{align}
for some $\rho_k\in [0,1)$. For $\rho_k=0$ we have $s^{k}= x^{k+1}$ and \cref{eq:solodov_bregman_alm} can be written as the exact Bregman proximal augmented Lagrangian method which corresponds to
\begin{align}
(x^{k+1}, y^{k+1}) = \argminimax_{x \in \bR^n, y \in \bR^m} ~L(x, y) +\tfrac{1}{2\sigma_k}\|x-x^k\|^2- \tfrac{1}{\sigma_k}D_\phi(y, y^k).
\end{align}
Under the assumption that the maximization wrt. $y$ is closed form, the multiplier variable $y$ can be marginalized out. Then the $x$-update boils down to the minimization of the marginal function called the $\sigma$-augmented Lagrangian $L_\sigma(x,y)$ given as
\begin{align}
L_{\sigma}(x, y) &:= \sup_{\eta \in \bR^m} L(x, \eta) - \tfrac{1}{\sigma}D_\phi(\eta, y),
\end{align}
which is a smooth function given that $f$ is smooth as established in the following lemma:
\begin{lemma} \label{thm:bregman_augmented_lagrangian}
   Let $\intr (\dom \phi) \cap \dom g^* \neq \emptyset$. Then the $\sigma$-augmented Lagrangian $L_\sigma(x,y)$ can be written as
\begin{align} \label{eq:augmented_lagr_composite}
L_{\sigma}(x, y) 
&=f(x) + \tfrac{1}{\sigma}\mathcal{P}_\sigma(\nabla \phi(y) + \sigma \mathcal{A}(x)) -\tfrac{1}{\sigma}\phi^*(\nabla \phi(y)),
\end{align}
for $\mathcal{P}_\sigma=\phi^* \infconv \sigma \star g$ which is continuously differentiable with $\nabla \mathcal{P}_\sigma = (\nabla \phi + \sigma \partial g^*)^{-1}$.
Furthermore, it holds for the maximizer
\begin{align}
y^+(x) =\argmax_{\eta \in \bR^m} \; L(x, \eta) - \tfrac{1}{\sigma} D_\phi(\eta, y)= (\nabla \phi + \sigma \partial g^*)^{-1}(\nabla \phi(y) + \sigma \mathcal{A}(x))
\end{align}
and we have the relation $v \in \partial_x L(x, y^+) \Leftrightarrow v \in \partial_x L_{\sigma}(x, y)$.
\end{lemma}
\begin{proof}
For brevity we omit the proof and point out that smoothness of infimal convolutions follows by combining \cite[Proposition 15.7]{BaCo110} and \cite[Proposition 18.8]{BaCo110}.
\end{proof}
Using the previous lemma the primal and dual updates read as:
\begin{subequations}
\begin{equation} \label{eq:subproblem_exact_alm}
x^{k+1} = \argmin_{x \in \bR^n} ~L_{\sigma_k}(x, y^k)+\tfrac{1}{\sigma_k}D_\psi(x,x^k)
\end{equation}
\begin{equation} \label{eq:dual_update}
y^{k+1} = (\nabla \phi + \sigma_k \partial g^*)^{-1}(\nabla \phi(y^k) + \sigma \mathcal{A}(x^{k+1})),
\end{equation}
\end{subequations}
where the $x$-update involves the minimization of the smooth Bregman augmented Lagrangian.

\begin{algorithm}[t!]
\caption{Bregman proximal augmented Lagrangian method for convex convex-composite problem}
\label{alg:bregman_proximal_alm_with_errors}
\begin{algorithmic}
\REQUIRE Choose $\rho_k \in [0, 1)$ and a sequence of positive step-sizes $\{\sigma_k\}_{k=0}^\infty$. Let $x^0 \in \bR^n$ and $y^0 \in \intr \dom \phi$.
\FORALL{$k=0, 1, \dots$}
    \STATE Compute $s^k$ by solving \cref{eq:subproblem_inexact_alm} approximately such that \cref{eq:error_bound_solodov_alm} holds.
    \STATE $x^{k+1} \gets \nabla \psi^*(\nabla \psi(s^k) - \sigma_k \nabla J_k(s^k))$
    \STATE $y^{k+1} \gets (\nabla \phi + \sigma_k \partial g^*)^{-1}(\nabla \phi(y^k) + \sigma_k \mathcal{A}(s^k))$
\ENDFOR
\end{algorithmic}
\end{algorithm}
However, in a practical implementation of the Bregman proximal augmented Lagrangian method exact solutions of the $x$-update \cref{eq:subproblem_exact_alm} are impractical. Instead, one approximately solves the smooth subproblem
\begin{align} \label{eq:subproblem_inexact_alm}
    s^k \approx \argmin_{x \in \bR^n }\left\{ J_k(x)\equiv f(x) + \tfrac{1}{\sigma_k}\mathcal{P}_{\sigma_k}(\nabla \phi(y^k) + \sigma_k \mathcal{A}(x)) + \tfrac{1}{\sigma_k}D_\psi(x,x^k)\right\},
\end{align}
such that $\|\nabla J_k(s^k)\|$ is small.
In light of \cref{thm:bregman_augmented_lagrangian} we have that $v^k:=\nabla J_k(s^k) -\frac{1}{\sigma_k}(\nabla \psi(s^k) - \nabla \psi(x^k)) \in \partial_x L(s^k, y^{k+1})$.
Hence the $x$-update \cref{eq:correction_solodov_primal} can be reparametrized in terms of the gradients $\nabla J_k(s^k)$ and the error criterion reduces to a condition that involves $s^k$ only:
\begin{align}\label{eq:error_bound_solodov_alm}
D_\psi(s^{k},x^+(s^k))&\leq \rho_k D_\psi(s^{k},x^k) + \rho_k D_\phi(y^{+}(s^k), y^k) =: \rho_k B_k(s^k),
\end{align}
where $x^+(s^k) = \nabla \psi^*(\nabla \psi(s^k) - \sigma_k \nabla J_k(s^k))$ is a correction or extra-gradient step.
The complete algorithm is listed in \cref{alg:bregman_proximal_alm_with_errors}.

\subsection{Dual subsequential convergence without domain interiority restriction}
Next we prove a refinement of \cref{thm:convergence_interior}, by showing asymptotic dual convergence without the assumption that $y^\infty \in \intr (\dom \phi)$. For that purpose we consider quadratic primal proximal regularization. This allows us to control the sequence $\{v^k\}_{k=0}^\infty$ of dual perturbation vectors showing that its limit is $0$. We first prove the following lemma:
\begin{lemma} \label{thm:solodov_prox_lemma_alm}
Let $\phi \in \Gamma_0(\bR^n)$ be Legendre and $\psi=\frac{1}{2}\|\cdot\|$. Then $\{s^k\}_{k=0}^\infty$ is bounded.
Further let $z^\star=(x^\star, y^\star) \in T^{-1}(0) \subseteq \bR^n \times \dom \phi$. Then $v^{k} \to 0$ and $\langle u^{k},y^{k+1} - y^\star \rangle \to 0$.
\end{lemma}
\begin{proof}
In light of \cref{thm:fejer_prox_solodov_implications:sum} we have that $\lim_{k\to \infty} \|s^{k} -  x^{k}\|= 0$ and $\lim_{k\to \infty} \|s^{k} -x^{k+1}\|=0$. In view of \cref{thm:fejer_prox_solodov_implications:bounded}, $\{x^k\}_{k=0}^\infty$ is bounded and thus $\{s^k\}_{k=0}^\infty$ is bounded as well.
Thus we have that $0\leq \sigma_k\|v^{k}\|=\|x^{k+1} - x^k\| \leq \|x^{k+1} - s^{k}\| +\| s^{k}- x^k\| \to 0$ and since $0< \sigma \leq \sigma_k$ is bounded away from $0$ we also have that $v^k \to 0$.
Since $x^{k}$ is bounded this implies that $\langle v^{k}, x^{k} - x^\star \rangle \to 0$.
Thanks to \cref{thm:fejer_prox_solodov_implications:sum_inner} we have that 
$
\langle v^{k}, s^{k} - x^\star \rangle + \langle u^{k}, y^{k+1}-y^\star \rangle \to 0,
$
and thus 
$
\langle u^{k}, y^{k+1}-y^\star \rangle \to 0,
$
as claimed.
\end{proof}
To prove asymptotic dual convergence we follow the proof \cite[Proposition 6]{eckstein2003practical} invoking the Rockafellian perturbation framework \cite{RoWe98} for convex duality. To this end we define the primal perturbation function:
\begin{align}
F(x, u)= f(x) + g(\mathcal{A}(x) + u).
\end{align}
and its conjugate which amounts to the partial convex conjugate of the Lagrangian,
\begin{align}
F^*(v,y) &=\sup_{x\in \bR^n} \langle x, v \rangle -L(x,y), \\
&=f^*(v -A^\top y) +\langle b,y \rangle +g^*(y),
\end{align}
and hence $y^\star$ is a dual solution if and only if $y^\star\in \argmin F^*(0, \cdot)$.

\begin{theorem}[dual asymptotic convergence]\label{thm:dual_asymototic_convergence_solodov}
Let $\phi \in \Gamma_0(\bR^n)$ be Legendre and $\psi=\frac{1}{2}\|\cdot\|$. Let $y^\star \in \argmin F^*(0, \cdot)$ be a dual solution.
Assume that $\{y^{k+1}\}_{k=0}^\infty$ is bounded which happens to be true if $D_\phi(y^\star, \cdot)$ is coercive. Then we have that $F^*(v^k, y^{k+1}) \to F^*(0, y^\star)$. In addition every limit point $y^\infty$ of the sequence of dual iterates $\{y^{k+1}\}_{k=0}^\infty$ is a dual solution, i.e., $y^\infty \in \argmin F^*(0, \cdot)$.
\end{theorem}
\begin{proof}
It can be readily checked that $(v^{k},u^{k}) \in T(s^{k}, y^{k+1})$ if and only if $(s^k, u^k) \in \partial F^*(v^k, y^{k+1})$. 
Let $(x^\star, y^\star) \in T^{-1}(0)$. Expanding the subgradient inequality at $(0, y^\star)$ we obtain
$$
F^*(0, y^\star) \geq F^*(v^k, y^{k+1}) + \langle s^k, 0 - v^k\rangle + \langle u^k, y^\star- y^{k+1} \rangle.
$$
Rearranging yields
\begin{align} \label{eq:inequality_dual_asymototic_convergence_solodov2}
F^*(v^k,y^{k+1}) \leq F^*(0, y^\star) +\langle s^k, v^k\rangle - \langle u^k, y^\star- y^{k+1} \rangle.
\end{align}
In light of \cref{thm:solodov_prox_lemma_alm} we have that $v^k\to 0$, $\{s^k\}_{k=0}^\infty$ is bounded and $\langle u^k, y^\star- y^{k+1} \rangle \to 0$. Then \cref{eq:inequality_dual_asymototic_convergence_solodov2} yields that
$
\limsup_{k \to \infty} ~F^*(v^k,y^{k+1}) \leq F^*(0, y^\star).
$
We now consider $\liminf_{k \to \infty}~ F^*(v^k, y^{k+1})$ which by the above is at most $F^*(0, y^\star)$.
Consider a subsequence indexed by $k_j$ which realizes the $\liminf$ which is possibly $-\infty$: 
$
\lim_{j \to \infty} F^*(v^{k_j}, y^{k_j+1}) = \liminf_{k \to \infty}~ F^*(v^k, y^{k+1}).
$
Since $y^{k_j+1}$ is bounded we can consider a subsequence if necessary to make sure that $y^{k_j+1} \to y^\infty$. 
Since $y^\star \in \argmin F^*(0, \cdot)$ we have
\begin{align*}
F^*(0,y^\star) &\leq 
F^*(0, y^\infty) \leq \lim_{j \to \infty} F^*(v^{k_j}, y^{k_j+1})=
\liminf_{k \to \infty}~ F^*(v^k, y^{k+1}),
\end{align*}
where the second inequality follows by lower semi-continuity.
Hence we have
$
\lim_{k \to \infty} F^*(v^k, y^k) = F^*(0, y^\star).
$
Let $y^\infty$ be any limit point of the sequence $\{r^{k}\}_{k=0}^\infty$ with $y^{k_j+1} \to y^\infty$ for some subsequence $\{y^{k_j+1}\}_{j=0}^\infty$. Then we have $F^*(0, y^\star) \leq F^*(0, y^\infty) \leq \liminf_{j \to \infty} F^*(v^{k_j}, y^{k_j+1}) = F^*(0, y^\star)$ and thus $y^\infty$ is a dual solution, i.e., $y^\infty \in \argmin F^*(0, \cdot)$. 
\end{proof}

\subsection{Ergodic analysis refined}
In this section we refine the ergodic analysis carried out in \cref{sec:ergodic_ppa} by considering the primal-dual gap. Notably, the primal-dual gap is different from the restricted gap function used for the general operator case.

We remind that $\dom f \cap \intr \dom \psi \neq \emptyset$ and $\dom g^*\cap \intr \dom \phi \neq \emptyset$. In fact, this is implied by the condition $\dom T \cap \intr \dom \Phi \neq \emptyset$.

\begin{lemma} \label{thm:ergodic_pd_fixed}
    Let $\psi \in \Gamma_0(\bR^n)$ and $\phi \in \Gamma_0(\bR^m)$ be Legendre. Then it holds that for any $x \in \dom \psi$ and $y \in \dom \phi$
\begin{align}
    L(s^{k}, y) -L(x, y^{k+1})&\leq \tfrac{1}{\sigma_k}\big(D_\Phi(z,z^k) -  D_\Phi(z,z^{k+1}) - (1-\rho_k) D_\Phi(p^k,z^{k})\big).
\end{align}
\end{lemma}
\begin{proof}
By \cref{eq:alm_solodov_1} we have that $v^{k} \in \partial_x L(s^{k}, y^{k+1})$ and $u^{k} \in \partial_y (-L)(s^{k}, y^{k+1})$. 
Let $y \in \dom \phi$ and $x \in\dom \psi$. By convexity of $L$ in the first argument we have that
\begin{align} \label{eq:ergodic_pd_upper}
L(x, y^{k+1}) &\geq L(s^{k}, y^{k+1}) + \langle v^{k}, x - s^{k} \rangle.
\end{align}
By convexity of $-L$ in the second argument we have that
\begin{align} \label{eq:ergodic_pd_lower}
-L(s^{k}, y) &\geq -L(s^{k}, y^{k+1}) + \langle u^{k}, y - y^{k+1} \rangle.
\end{align}
Summing \cref{eq:ergodic_pd_upper,eq:ergodic_pd_lower} yields
\begin{align}
    L(s^{k}, y) -L(x, y^{k+1})&\leq \langle v^{k},  s^{k}-x \rangle + \langle u^{k}, y^{k+1}-y \rangle.
\end{align}
Identifying $w^k=(v^k,u^k) \in T(p^k)$, $p^k=(s^k, y^{k+1})$ and $z=(x,y)$ we obtain invoking \cref{thm:ergodic_four_point} 
\[
    L(s^{k}, y) -L(x, y^{k+1})\leq \tfrac{1}{\sigma_k}\big(D_\Phi(z,z^k) -  D_\Phi(z,z^{k+1}) - (1-\rho_k) D_\Phi(p^k,z^{k})\big) \qedhere
\]
\end{proof}

The ergodic primal and dual iterates for $0 \leq K$ are:
\begin{align}
    \breve s^K := \frac{\sum_{k=0}^{K} \sigma_k s^{k}}{\sum_{k=0}^{K} \sigma_k} \quad \text{and} \quad \breve y^K:= \frac{\sum_{k=0}^{K} \sigma_k y^{k+1}}{\sum_{k=0}^{K} \sigma_k}.
\end{align}

The next result shows that every limit point of the ergodic sequence is a saddle-point of $L$. This is even valid for Burg's entropy whose domain is open.
\begin{proposition} \label{thm:ergodic_pd}
Let $\psi \in \Gamma_0(\bR^n)$ and $\phi \in \Gamma_0(\bR^m)$ be Legendre. Then we have for any $x \in \dom \psi$ and any $y \in \dom \phi$:
$$
L(\breve s^K, y) - L(x, \breve y^K) \leq \frac{1}{\sum_{k=0}^{K} \sigma_k} \big(D_\psi(x, x^0) + D_\phi(y, y^0)\big).
$$
If, in addition, $\dom f \subseteq \cl \dom \psi$ and $\dom g^* \subseteq \cl \dom \phi$ every limit point $(\breve s^\infty, \breve y^\infty)$ of the sequence of ergodic iterates $\{\breve s^K, \breve y^K)\}_{K=0}^\infty$ is a saddle-point of $L$.
\end{proposition}
\begin{proof}
Since $L$ is convex in $x$ and concave in $y$ we have in view of \cref{thm:ergodic_pd_fixed} that for any $x \in \dom \psi$ and $y \in \dom \phi$
\begin{align*}
L(\breve s^K, y) - L(x, \breve y^K) &\leq \frac{1}{\sum_{k=0}^{K} \sigma_k} \sum_{k=0}^{K} \sigma_k \big( L(s^{k}, y) - L(x, y^{k+1}) \big) \\
&\leq \frac{1}{\sum_{k=0}^{K} \sigma_k} \sum_{k=0}^{K} \big(D_\Phi(z,z^k) -  D_\Phi(z,z^{k+1}) - (1-\rho_k) D_\Phi(p^k,z^{k})\big) \\
&\leq \frac{1}{\sum_{k=0}^{K} \sigma_k} \sum_{k=0}^{K} \big(D_\Phi(z,z^k) -  D_\Phi(z,z^{k+1}) \big) \\
&\leq \frac{1}{\sum_{k=0}^{K} \sigma_k} \big(D_\psi(x, x^0)+ D_\phi(y, y^{0}) \big).
\end{align*}
Now consider a convergent subsequence $(\breve s^{K_j}, \breve y^{K_j}) \to (\breve s^\infty, \breve y^\infty)$ for $j \to \infty$.
Let $x \in \dom \psi$ and $y \in \dom \phi$.
Since $L(\cdot, y)$ and $-L(x, \cdot)$ are both lsc we have that
\begin{align*}
L(\breve s^\infty, y) - L(x, \breve y^\infty) &\leq \liminf_{j \to \infty} L(\breve s^{K_j}, y) - L(x, \breve y^{K_j}) \leq 0.
\end{align*}
Let $x \in \dom f \subseteq \cl \dom \psi$ and $y \in \dom g^*\subseteq \cl \dom \phi$. Take $\bar x \in \intr (\dom \psi) \cap \dom f \neq \emptyset$ and $\bar y \in \intr (\dom \phi) \cap \dom g^* \neq \emptyset$ and consider $x_\lambda:=(1-\lambda) \bar x + \lambda x$ and $y_\lambda:=(1-\lambda) \bar y + \lambda y$ for $\lambda \in [0, 1)$
Then by \cite[Theorem 6.1]{Roc70}, $x_\lambda \in \intr( \dom \psi ) \cap \dom f$ and $y_\lambda \in \intr( \dom \phi ) \cap \dom g^*$ and we have
$$
L(\breve s^\infty, y_\lambda) \leq L(x_\lambda, \breve y^\infty).
$$
By concavity in $y$ and convexity in $x$ we have
\begin{align}
\lambda L(\breve s^\infty, y) + (1-\lambda) L(\breve s^\infty, \bar y)  &\leq L(\breve s^\infty, y_\lambda)\notag \\
&\leq L(x_\lambda, \breve y^\infty) \leq \lambda L(x, \breve y^\infty) + (1-\lambda) L( \bar x, \breve y^\infty)
\end{align}
Passing $\lambda \to 1$ we obtain
$$
 L(\breve s^\infty, y) \leq  L(x, \breve y^\infty),
$$
and thus $(\breve s^\infty, \breve y^\infty)$ is a saddle-point of $L$.
\end{proof}

In what follows we specialize to problems with conic constraints. To this end prove a generalization of \cite[Theorem 4]{xu2021iteration} and \cite[Theorem 3.1]{yan2020bregman} for our inexact Bregman augmented Lagrangian framework with conic constraints.
\begin{proposition}[asymptotic feasibility for conic constraints]
Let $\psi \in \Gamma_0(\bR^n)$ and $\phi \in \Gamma_0(\bR^m)$ be Legendre. Let $g=\delta_{-\cC}$ for a closed and convex cone $\mathcal{C} \subseteq \bR^m$ such that $\cC^* \subseteq \dom \phi$. Let $(x^\star, y^\star) \in \dom \Phi$ be a saddle-point of $L$.
Then we have that
$$
\max\{|f(\breve s^K) - f(x^\star)|, \dist(\mathcal{A}(\breve s^K), -\cC)\} \leq \frac{1}{\sum_{k=0}^{K} \sigma_k} \big(D_\psi(x^\star,x^0) + \max_{y \in \cC^* \cap B_R} D_\phi(y, y^0)\big),
$$
where $R:=2\|y^\star\| + 1$.
\end{proposition}
\begin{proof}
Let $y \in \cC^*$. Since $\cC^*$ is convex and $\{y^{k}\}_{k=0}^\infty \subset \cC^*$ we have that $\breve y^{K} \in \cC^*$ and since $\mathcal{A}(x^\star) \in -\cC$ we obtain that $\langle \breve y^{K}, \mathcal{A}(x^\star) \rangle \leq 0$. Thus we have
\begin{align*}
L(\breve s^K, y) - L(x^\star, \breve y^K) &= f(\breve s^K) - f(x^\star) + \langle y, \mathcal{A}(\breve s^K)\rangle - \langle \breve y^{K}, \mathcal{A}(x^\star) \rangle \\
&\geq f(\breve s^K) - f(x^\star) + \langle y, \mathcal{A}(\breve s^K)\rangle.
\end{align*}
Invoking \cref{thm:ergodic_pd_fixed} we obtain
\begin{align}
\frac{1}{\sum_{k=0}^{K} \sigma_k} \big(D_\psi(x^0,x^\star) + D_\phi(y, y^0)\big) &\geq L(\breve s^K, y) - L(x^\star, \breve y^K) \notag \\
&\geq f(\breve s^K) - f(x^\star) + \langle y, \mathcal{A}(\breve s^K)\rangle\label{eq:bound_primal_dual}.
\end{align}
For $y =0$ we have that
\begin{align} \label{eq:ergodic_simple}
    f(\breve s^K) - f(x^\star) \leq \frac{1}{\sum_{k=0}^{K} \sigma_k} \big(D_\psi(x^0,x^\star) + D_\phi(0, y^0)\big).
\end{align}
Let $R:= 2\|y^\star\| + 1$.
Taking the maximum over $y \in \cC^* \cap B_R$ on both sides of the inequality \cref{eq:bound_primal_dual} noting that
\begin{align*}
\max_{y \in \cC^* \cap B_R} \langle y, u\rangle = (\delta_{\cC^*} + \delta_{B_R})^*(u)  =(\delta_{-\cC} \infconv (R\|\cdot\|))(u) =R \dist(u,-\cC)
\end{align*}
we obtain
\begin{align}
    \frac{1}{\sum_{k=0}^{K} \sigma_k} \big(D_\psi(x^0,x^\star) + \max_{y \in \cC^* \cap B_R} D_\phi(y, y^0)\big) \geq f(\breve s^K) - f(x^\star) + R \dist(\mathcal{A}(\breve s^K), -\cC).\label{eq:ergodic_1}
\end{align}
Since $(x^\star,y^\star)$ is a saddle-point of $L$ we have for all $x \in \bR^n$ and $y \in \cC^*$ that
$$
f(x) + \langle \mathcal{A}(x), y^\star \rangle-f(x^\star) - \langle \mathcal{A}(x^\star), y\rangle= L(x, y^\star) - L(x^\star, y) \geq 0
$$
and thus for $y=0$ and $x=\breve s^K$ since $y^\star \in \cC^*$
\begin{align} 
0 &\geq f(x^\star)-f(\breve s^K) -\langle y^\star, \mathcal{A}(\breve s^K) \geq f(x^\star)-f(\breve s^K) - \max_{y \in \cC^*\cap \|y^\star\|\cB} \langle y, \mathcal{A}(\breve s^K)\rangle \notag \\
&= f(x^\star)-f(\breve s^K) - \|y^\star\|\dist(\mathcal{A}(\breve s^K),-\cC) \label{eq:ergodic_2}
\end{align}
Summing \cref{eq:ergodic_1} and \cref{eq:ergodic_2} since $R=2\|y^\star\|+1$ we obtain
\begin{align}
\dist(\mathcal{A}(\breve s^K),-\cC) &\leq (\|y^\star\| + 1)\dist(\mathcal{A}(\breve s^K),-\cC) \notag \\
&\leq \frac{1}{\sum_{k=0}^{K} \sigma_k} \big(D_\psi(x^0,x^\star) + \max_{y \in \cC^* \cap B_R} D_\phi(y, y^0)\big) \label{eq:final_primal_dual}.
\end{align}
From \cref{eq:ergodic_2} and \cref{eq:final_primal_dual} we obtain
\begin{align*}
    f(x^\star)-f(\breve s^K) &\leq \|y^\star\| \dist(\mathcal{A}(\breve s^K),-\cC) \\
    &\leq \frac{1}{\sum_{k=0}^{K} \sigma_k} \big(D_\psi(x^0,x^\star) + \max_{y \in \cC^* \cap B_R} D_\phi(y, y^0)\big).
\end{align*}
Together with \cref{eq:ergodic_simple} we obtain
\[
|f(x^\star)-f(\breve s^K)| \leq \frac{1}{\sum_{k=0}^{K} \sigma_k} \big(D_\psi(x^0,x^\star) + \max_{y \in \cC^* \cap B_R} D_\phi(y, y^0)\big) \qedhere
\]
\end{proof}

\subsection{Path-following and second-order oracles}
\subsubsection{Newton oracle and generalized self-concordance}
In this section we consider the solution of the subproblem~\cref{eq:subproblem_inexact_alm} with a Newton oracle. The combined algorithm is listed in \cref{alg:palm_newton}. We consider choices for the primal and dual regularization $\psi$ and $\phi$ encoding the geometries of $f$ and $g^*$ respectively such that the resulting subproblems are either \emph{self-concordant} (SC) in the classical sense \cite{nesterov1994interior,} or \emph{quasi self-concordant} (QSC) \cite{bach2010self}. 
These notions have been unified under the term \emph{generalized self-concordance} \cite{sun2019generalized} which is utilized in recent related work \cite{dvurechensky2019generalized}.
In our case the path following mechanism is implemented via a careful selection of the step-size in the proximal point scheme which guarantees that the initial point is in the region of quadratic convergence.
The local convergence for the latter case was recently established by \cite{doikov2025minimizing}.
\begin{algorithm}[t!]
\caption{Bregman Proximal Augmented Lagrangian Method with Newton Oracle}
\label{alg:palm_newton}
\begin{algorithmic}[1]
\REQUIRE Let $x^0\in \intr \dom \psi$ and $y^0 \in \intr \dom \phi$. Choose $\rho_k \in [0, 1)$ and a sequence of positive step-sizes $\{\sigma_k\}_{k=0}^\infty$.
\FORALL{$k=0, 1, \dots$}
    \STATE Initialize $s^{k,0}\gets x^k$
    \FORALL{$t=0, 1, \dots$}
        \STATE $s^{k,t+1} \gets s^{k,t} - (\nabla^2 J_k(s^{k,t}))^{-1}\nabla J_k(s^{k,t})$
        \IF{$D_\psi(x^+(s^{k,t+1}), s^{k,t+1}) \leq \rho_k B_k(s^{k,t+1})$}
        \STATE $s^k \gets s^{k,t+1}$
        \STATE \textbf{goto} \ref{state:x}
        \ENDIF
    \ENDFOR
    \STATE \label{state:x}$x^{k+1} \gets \nabla \psi^*(\nabla \psi(s^k) - \sigma_k \nabla J_k(s^k))$
    \STATE $y^{k+1} \gets (\nabla \phi + \sigma_k\partial g^*)^{-1}(y^k + \sigma_k (As^k - b))$
\ENDFOR
\end{algorithmic}
\end{algorithm}

\subsubsection{Quasi self-concordant penalties}
In this section we specialize the primal regularization as $\psi(x)=\frac{1}{2}\|x\|^2$. In this case, since $x^+(s^k)=s^k - \sigma_k \nabla J_k(s^k)$, the stopping criterion~\cref{eq:error_bound_solodov_alm} simplifies to:
\begin{align}\label{eq:error_bound_solodov_alm_quadratic}
\tfrac{\sigma_k^2}{2}\| \nabla J_k(s^k)\|^2&\leq  \tfrac{\rho_k}{2}\|s^{k}-x^k\|^2 + \rho_k D_\phi(y^{+}(s^k), y^k) = \rho_k B_k(s^k).
\end{align}
We choose the dual regularization $\phi$ as either the entropy \cref{ex:neumann} or in terms of an unconventional construction based on Spence's function \cref{ex:spence}.

This leads to subproblems~\cref{eq:subproblem_inexact_alm} whose objective $J_k$ exhibits \emph{quasi self-concordance} (QSC) \cite{bach2010self,doikov2025minimizing} defined next:
\begin{definition}[quasi self-concordant functions]
Let $f :\bR^n \to \bR$ be convex and $\mathcal{C}^3$. Then we say that $f$ is \emph{quasi self-concordant} (QSC) with constant $M$ if for all $x \in \bR^n$
\begin{align}
| D^3 f(x)[u,u, v] | \leq M \langle u, \nabla^2 f(x) u\rangle \|v\|
\end{align}
for any $u,v\in \bR^n$.
\end{definition}
Next we provide examples for combinations of $g$ and $\phi$ that lead to QSC penalty functions $\mathcal{P}_\sigma$ in Bregman augmented Lagrangian function.
\begin{example}[exponential penalty] \label{ex:exp}
Let $g:=\delta_{-\bR^m_+}$ and $\phi$ as in \cref{ex:neumann}. Then $\mathcal{P}_\sigma=\phi^* \infconv \sigma \star g=\phi^*=\sumexp$ for $\sumexp(u)=\sum_{i=1}^m \exp(u_i)$ is the exponential penalty function used in the exponential method of multipliers \cite{tseng1993convergence}.
 In light of \cite{doikov2025minimizing}, $\sumexp$ is QSC with constant $1$.
\end{example}
Note that the augmented Lagrangian function is also anisotropically smooth \cite[Definition 3.1]{laude2025anisotropic} and can thus be minimized effectively with the \emph{anisotropic proximal gradient algorithm} \cite{laude2025anisotropic}, a nonlinearly preconditioned version of the proximal gradient method.

\begin{example}[softmax penalty] \label{ex:softmax}
    Choose $g(u):=\vecmax(u)=\max\{u_1, u_2, \ldots, u_m\}$ and $\phi$ as in \cref{ex:neumann}. It can be readily checked via conjugate duality that $\mathcal{P}_\sigma=\phi^* \infconv \sigma \star g=\logsumexp+1$ with $\logsumexp:=\ln \circ \sumexp$.
    In light of \cite{doikov2025minimizing}, $\logsumexp$ is QSC with constant $2$.
\end{example}
Via conjugate duality, \cref{ex:spence} yields a penalty function which is QSC and Lipschitz smooth and therefore allows for efficient optimization with both first and second-order methods.
Notably, it approximates the classical max-quadratic penalty function used in the classical augmented Lagrangian method:
\begin{example}[softplus penalty] \label{ex:softplus}
  Let $g:=\delta_{-\bR^m_+}$ and $\phi$ as in \cref{ex:spence}. Then $\mathcal{P}_\sigma=\phi^* \infconv \sigma \star g=\phi^*$ and in light of \cref{ex:spence}, $(\varphi^*)'(t)=\ln(1+\exp(t))$ is the softplus function. Define the sigmoid function $p(t):=\frac{e^t}{1+e^t}\in(0,1)$. Then $(\varphi^*)''(t)=p(t)$ and $(\varphi^*)'''(t)=p(t)(1-p(t))$.
Since $0<1-p(t)<1$, we have the pointwise bound $|(\varphi^*)'''(t)| = p(t)(1-p(t))\leq p(t) =(\varphi^*)''(t)$ and consequently $\mathcal{P}_\sigma$ is QSC with constant $1$.
\end{example}

\begin{lemma} \label{thm:qsc}
Let $f$ be convex and quasi self-concordant with modulus $M_f$. Let $\mathcal{P}_{\sigma_k}$ be quasi self-concordant with modulus $\alpha_k$. Then $J_k$ is quasi self-concordant with modulus $M_k=\max\{\sigma_k \alpha_k \|A\|, M_f\}$ and strongly convex with modulus $\mu_k=\frac{1}{\sigma_k}$. 
\end{lemma}
\begin{proof}
This follows from \cite{doikov2025minimizing}.
\end{proof}

\begin{proposition} \label{thm:complexity_superlinear_qsc}
    Let $f$ be convex and quasi self-concordant with modulus $M_f$. Let $\mathcal{P}_{\sigma_k}$ be quasi self-concordant with modulus $\alpha_k$. Denote by $F_k(x):=f(x)+\frac{1}{\sigma_k}\mathcal{P}_{\sigma_k}(\nabla \phi(y^k) + \sigma_k \mathcal{A}(x))$ and $g_k := \|\nabla J_k(x^k)\|=\|\nabla F_k(x^k)\|$. Choose
\begin{equation} \label{eq:choice_tauk}
    \sigma_k \leq \min\Big\{\tfrac{1}{2 g_k M_f}, \tfrac{1}{\sqrt{2 g_k \alpha_k \|A\|}} \Big\}
\end{equation}
Then, after
\begin{align}
    T_k = \left\lceil \log_2  \ln \tfrac{1}{4 \max\{\sigma_k \alpha_k \|A\|, M_f\}\exp(-1) \sqrt{2 \rho_k}\sqrt{B_k(s^k)}} \right\rceil
\end{align}
pure Newton steps on $J_k$ starting from $x^k$, the stopping criterion~\cref{eq:error_bound_solodov_alm_quadratic} is satisfied.
\end{proposition}
\begin{proof}
In light of \cref{thm:qsc}, $J_k$ is quasi self-concordant with modulus $M_k=\max\{\sigma_k \alpha_k \|A\|, M_f\}$ and strongly convex with modulus $\mu_k=\frac{1}{\sigma_k}$.
In addition we have for $\sigma_k$ that
$
  \sigma_k \leq \frac{1}{2 g_k M_k}.
$
This means that $s^{k,0} = x^k$ lies in the set $U=\{x \in \bR^n : \|\nabla J_k(x^k)\| \leq \tfrac{\mu_k}{2M_k}\}$.
By adapting the proof of \cite[Theorem 8]{doikov2025minimizing} a pure Newton step obeys:
\begin{align} 
\|\nabla J_k(s^{k,t+1})\| \leq \tfrac{2M_k}{\mu_k}\exp(-1)\|\nabla J_k(s^{k,t})\|^2 = 2M_k\exp(-1) \sigma_k \|\nabla J_k(s^{k,t})\|^2.
\end{align}
After $T_k$ steps we thus have:
\begin{align} 
\|\nabla J_k(s^{k,T_k})\| \leq \frac{(\theta_k)^{2^{T_k}}}{c_k},
\end{align}
for $\theta_k =c_k g_k$ and $c_k=  2M_k\exp(-1) \sigma_k$.
Hence the stopping criterion~\cref{eq:error_bound_solodov_alm_quadratic} is satisfied for any $T_k$ sufficiently large such that
\begin{align}
    (\theta_k)^{2^{T_k}} \leq 4 M_k\exp(-1) \sqrt{2 \rho_k}\sqrt{B_k(s^k)}.
\end{align}
Taking the logarithm on both sides, multiplying with $-1$ and taking the logarithm with base $2$ we obtain that the inequality is implied if
\begin{align}
T_k\geq \log_2 \frac{\ln \tfrac{1}{4 M_k\exp(-1) \sqrt{2 \rho_k}\sqrt{B_k(s^k)}}}{\ln \tfrac{1}{\theta_k}}
\end{align}
Since $\sigma_k \leq \frac{1}{2 g_k M_k}$, $\theta_k\leq \exp(-1)$ and thus $\ln\frac1{\theta_k} \geq 1$. Hence, the inequality is implied if
\[
    T_k \geq \log_2  \ln \tfrac{1}{4 M_k\exp(-1) \sqrt{2 \rho_k}\sqrt{B_k(s^k)}} \qedhere
\]
\end{proof}
\begin{remark}
We remark that $g_k(\sigma_k)$ in \cref{eq:choice_tauk} depends on $\sigma_k$ itself and therefore a backtracking or bisection procedure has to be invoked to compute a feasible step-size $\sigma_k$ small enough. Such a scheme is well-defined by continuity and the fact that 
\begin{align}
g_k(\sigma_k)=\|\nabla f(x^k) + A^\top (\nabla \phi + \sigma_k \partial g^*)^{-1}(\nabla \phi(y^k) + \sigma_k \mathcal{A}(x^k))\| \to \|\nabla f(x^k) +A^\top y^k\|,
\end{align}
as $\sigma_k \to 0$.
Furthermore, as the algorithm progresses $g_k \to 0$ which allows us to drive $\sigma_k \to \infty$ eventually, thereby ensuring super-linear convergence of the outer loop if an error bound holds at the solution.
\end{remark}
\begin{remark}
Note that $B_k(s^k)$ serves as a proxy for the ``magnitude'' of the proximal point residual and thus it can be used as a termination criterion of the outer loop. In practice we also cap $\rho_k$ by a small constant. Thanks to the double-logarithm, the number of inner Newton steps can thus be safely upper bounded by a small constant, say $T_k \leq 10$.
\end{remark}
Next we show that if in addition $\nabla J_k$ is Lipschitz smooth, which happens to be the case if $\mathcal{P}_{\sigma_k}$ is chosen as in \cref{ex:softmax} or \cref{ex:softplus}, the dependence on $B_k(s^k)$ can be eliminated. In particular for fixed $\sigma_k,\rho_k$, the maximum number of Newton steps $T_k$ is a computable constant, potentially much smaller than $10$.
\begin{proposition}
In the situation of \cref{thm:complexity_superlinear_qsc} assume that $f$ is $L_f$-Lipschitz smooth and $\mathcal{P}_{\sigma_k}$ is $\beta_k$-Lipschitz smooth as in \cref{ex:softmax} or \cref{ex:softplus}. Then $J_k$ is Lipschitz-smooth with constant $L_k= L_f + \beta_k\sigma_k \|A\|^2 + \frac{1}{\sigma_k}$. Further, after
\begin{align}
    T_k = \left\lceil \log_2 (\ln (\sqrt{2}L_k \sigma_k + \sqrt{\rho_k} ) - \ln \sqrt{\rho_k} + 1) \right\rceil
\end{align}
pure Newton steps on $J_k$ starting from $x^k$, the stopping criterion~\cref{eq:error_bound_solodov_alm_quadratic} is satisfied.
\end{proposition} 
\begin{proof}
Let $s \in \bR^n$. Thanks to $L_k$-Lipschitz continuity of $\nabla J_k$ we have the bound
\begin{align}
g_k &= \|\nabla J_k(x^k)\| \notag\\
&= \|\nabla J_k(x^k) - \nabla J_k(s) + \nabla J_k(s)\| \notag\\
&\leq L_k \|s-x^k\| + \|\nabla J_k(s)\| \notag\\
&\leq L_k \sqrt{2 B_k(s)} +  \|\nabla J_k(s)\|.
\end{align}
and thus
\begin{align}
\sqrt{\rho_k} \sqrt{B_k(s)} \geq \tfrac{\sqrt{\rho_k}}{\sqrt{2} L_k} (g_k -\|\nabla J_k(s)\|).
\end{align}
Hence, the stopping criterion~\cref{eq:error_bound_solodov_alm_quadratic} for $s^{k,T_k}$ is guaranteed to hold if
\begin{align}
(\sigma_k + \tfrac{\sqrt{\rho_k}}{\sqrt{2}L_k}) \|\nabla J_k(s^{k,T_k })\|  \leq \tfrac{\sqrt{\rho_k}}{\sqrt{2} L_k} g_k.
\end{align}
Following the previous proof the stopping criterion~\cref{eq:error_bound_solodov_alm_quadratic} holds for any $T_k$ sufficiently large such that
\begin{align}
     (\sigma_k + \tfrac{\sqrt{\rho_k}}{\sqrt{2}L_k}) \frac{(\theta_k)^{2^{T_k}}}{c_k} \leq \frac{\sqrt{\rho_k}}{\sqrt{2}L_k} g_k = \frac{\sqrt{\rho_k}}{\sqrt{2}L_k} \frac{\theta_k}{c_k}.
\end{align}
Rewriting, taking the logarithm on both sides, multiplying with $-1$ and taking the logarithm with base $2$ on both sides yields
\begin{align}
     T_k  \geq \log_2 \left(\frac{\ln (\sqrt{2}L_k \sigma_k + \sqrt{\rho_k} ) - \ln \sqrt{\rho_k}}{-\ln \theta_k} + 1 \right)
\end{align}
Note that by the choice of $\sigma_k$, $\theta_k\leq \exp(-1)$ and thus $-\ln(\theta_k) \geq 1$ and therefore the bound is implied by
\[
T_k  \geq \log_2 (\ln (\sqrt{2}L_k \sigma_k + \sqrt{\rho_k} ) - \ln \sqrt{\rho_k} + 1) \qedhere
\]
\end{proof}

\subsubsection{Self-concordant regularization in the primal}
Unlike classical self-concordant functions, quasi self-concordant functions are not invariant under affine transformations. Consequently, when $g=\delta_{-\bR^m_+}$, it is natural to choose $\phi$ as a logarithmic barrier. However, the resulting stopping criterion~\cref{eq:error_bound_solodov_alm_quadratic} is expressed in the Euclidean norm, which cannot in general be controlled by the local norm underlying self-concordant analysis, namely the Newton decrement \begin{align}
\lambda(s)= M_k \sqrt{\langle \nabla J_k(s), \nabla^2 J_k(s)^{-1} \nabla J_k(s)\rangle}.
\end{align}
This mismatch between Euclidean and intrinsic geometries prevents a direct application of classical self-concordant arguments.

To address this issue, we reformulate problem~\cref{eq:problem} by introducing slack variables $\xi=\mathcal A(x)$. This reformulation permits barrier-type regularization on the slack variables. In particular, by using Burg’s entropy we align the stopping criterion~\cref{eq:error_bound_solodov_alm} with the local geometry induced by self-concordance.

In the remainder of this subsection we focus on the equality constrained case $g=\delta_{\{0\}}$ and assume that $f$ is self-concordant on the interior of its domain. The primal proximal regularization $\psi$ is chosen as a self-concordant Legendre function encoding the domain geometry of $f$ and the dual proximal regularization $\phi$ as a simple quadratic.

Self-concordant functions are defined following \cite[Definition 2.1.1]{nesterov1994interior} which reads after translation of constants $M=\frac{1}{\sqrt{a}}$:
\begin{definition}[self-concordance]
Let $f \in \Gamma_0(\bR^n)$ and $Q$ an open subset of $\dom f$. Then we say that $f$ is \emph{self-concordant} (SC) with constant $M > 0$ on $Q$ if $f \in \mathcal{C}^3(Q)$ and 
\begin{align}
|D^3 f(x)[u, u, u] |\leq 2M \langle u, \nabla^2 f(x)u \rangle^{3/2},
\end{align}
for every $u \in \bR^n$.
A self-concordant function with constant $M$ is called strongly self-concordant on $Q$ if the level-sets $\{x \in Q \mid f(x) \leq \alpha\}$ are closed for all $\alpha$.
\end{definition}
In light of \cite[Remark 2.1.1]{nesterov1994interior} a self-concordant function $f \in \Gamma_0(\bR^n)$ is automatically strongly self-concordant if $\dom f = Q$.

We first prove the following lemma which establishes a local equivalence between the Bregman distance in the stopping criterion~\cref{eq:error_bound_solodov_alm} and the local norm of the gradient:
\begin{lemma}[equivalence of Bregman distance and local norm] \label{thm:self_concordance_stopping}
Let $\psi \in \Gamma_0(\bR^n)$ be Legendre and strongly self-concordant on $\intr \dom \psi$ with constant $M_\psi$. Let $e \in \bR^n$ and $s \in \intr \dom \psi$. Asssume that $\sqrt{\langle e, \nabla^2 \psi(s)^{-1} e \rangle} < \frac{1}{3M_\psi}$
Then it holds for $x^+(s):=\nabla \psi^*(\nabla \psi(s) - e)$ that
\begin{align}
D_\psi(x^+(s), s)&\leq 2 \langle e, \nabla^2 \psi(s)^{-1} e \rangle.
\end{align}
\end{lemma}
\begin{proof}
Let $s\in \intr \dom \psi$. Using \cref{thm:dual_bregman:flip} and Taylor's theorem there is some $\theta \in [0,1]$ such that for 
$
v_\theta:=\nabla \psi(s) + \theta(\nabla \psi(x^{+}(s)) - \nabla \psi(s)) =\nabla \psi(s) - \theta e
$
we have
\begin{align} \label{eq:bregman_to_newton_decrement}
D_\psi(x^{+}(s), s) &= D_{\psi^*}(\nabla \psi(s),\nabla \psi(x^{+}(s))) \notag\\
&=\tfrac{1}{2}\langle \nabla \psi(x^+(s))- \nabla \psi(s), \nabla^2 \psi^*(v_\theta)(\nabla \psi(x^+(s))- \nabla \psi(s)) \rangle \notag\\
&=\tfrac{1}{2}\langle \nabla \psi(x^+(s))- \nabla \psi(s), \nabla^2 \psi(\nabla \psi^*(v_\theta))^{-1}(\nabla \psi(x^+(s))- \nabla \psi(s)) \rangle \notag\\
&=\tfrac{1}{2}\langle e, \nabla^2 \psi(\nabla \psi^*(v_\theta))^{-1} e \rangle,
\end{align}
where the third equality follows by the inverse Hessian identity \cref{thm:inverse_hessian}.
Define $\|w\|_{s}=\sqrt{\langle w, \nabla^2 \psi(s)w \rangle}$ and its dual $\|w^*\|_s^*=\sqrt{\langle w^*, \nabla^2 \psi(s)^{-1} w^* \rangle}$.
We first show that that
\begin{align} \label{nest_nem123}
r:=M_\psi \|\nabla \psi^*(v_\theta) - s\|_{s}  < \tfrac{1}{2},
\end{align}
is valid. Since $\psi$ is self-concordant on $\dom \psi$ with constant $M_\psi$, in light of \cite[Theorem 5.1.8]{nesterov2018lectures}, the following inequality holds for any $u,v \in \dom \psi$:
\begin{align} \label{eq:bound518}
 \frac{\|v-u\|_{u}^2}{1+M_\psi \|v-u\|_{u}} \leq \langle \nabla \psi(v)-\nabla \psi(u),v-u \rangle \leq \|v-u\|_u \cdot \|\nabla \psi(v)-\nabla \psi(u)\|_u^*.
\end{align}
Now choose $v:=\nabla \psi^*(v_\theta)$ and $u=s$. Then $\|v-u\|_{u}= \|\nabla \psi^*(v_\theta) - s\|_{s} = \frac{r}{M_\psi}$ and 
\begin{align}
\|\nabla \psi(v)-\nabla \psi(u)\|_u^* &=\|\nabla \psi(\nabla \psi^*(v_\theta))-\nabla \psi(s)\|_s^* =\|v_\theta -\nabla \psi(s)\|_s^* =\theta\|e\|_s^*
\end{align}
and thus \cref{eq:bound518} reads,
$
 \frac{r}{1 + r} \leq \theta M_\psi\|e\|_s^* \leq M_\psi\| e\|_s^*.
$
By assumption $M_\psi \| e\|_s^* < \frac{1}{3}$ and so $r < \frac{1/3}{2/3} = \tfrac{1}{2}$.
Since $\psi$ is self-concordant on its domain with constant $M_\psi$ and \cref{nest_nem123} holds we obtain by invoking \cite[Theorem 2.1.1(i)]{nesterov1994interior}:
\begin{align}  \label{eq:bregman_to_newton_decrement_hessian_bound}
\langle e, \nabla^2 \psi(\nabla \psi^*(v_\theta))^{-1} e \rangle &\leq \tfrac{1}{(1-r)^2}\langle e, \nabla^2 \psi(s)^{-1} e\rangle =4 \langle e, \nabla^2 \psi(s)^{-1} e \rangle.
\end{align}
Combining \cref{eq:bregman_to_newton_decrement,eq:bregman_to_newton_decrement_hessian_bound} we obtain
\begin{align}
D_\psi(x^{+}(s), s)&\leq 2 \langle e, \nabla^2 \psi(s)^{-1} e \rangle,
\end{align}
as claimed.
\end{proof}

\begin{lemma} \label{thm:sc}
Let $f\in \Gamma_0(\bR^n)$ be self-concordant on $\intr \dom f$ with constant $M_f$. Assume that $\psi\in \Gamma_0(\bR^n)$ is strongly self-concordant on $\dom \psi\subseteq \intr \dom f$ with constant $M_\psi$. Denote by $F_k(x):=f(x)+\langle y^k, \mathcal{A}(x)\rangle + \frac{\sigma_k}{2} \|\mathcal{A}(x)\|^2$. Then $J_k(x)=F_k(x)+\frac{1}{\sigma_k}D_\psi(x, x^k)$ is strongly self-concordant with constant $M_k=\max\{M_f, \sqrt{\sigma_k}M_\psi \}$ on $\dom J_k = \dom \psi$.
\end{lemma}
\begin{proof}
This follows by \cite[Theorem 5.1.1]{nesterov2018lectures} and \cite[Corollary 5.1.2]{nesterov2018lectures}.
\end{proof}

\begin{proposition} \label{thm:complexity_superlinear_sc}
Let $f\in \Gamma_0(\bR^n)$ be self-concordant on $\intr \dom f$ with constant $M_f$. Assume that $\psi\in \Gamma_0(\bR^n)$ is strongly self-concordant on $\dom \psi\subseteq \intr \dom f$ with constant $M_\psi$.
     Assume that $\nabla^2 F_k(s) \preceq c(\sigma_k) \nabla^2 \psi(s)$ for all $s \in \dom \psi$ and constant $c(\sigma_k)$. Choose
\begin{equation} \label{eq:choice_tauk_sc}
    \sigma_k < \frac{1}{16 M_k^2 \langle \nabla J_k(x^k), \nabla^2 \psi(x^k)^{-1} \nabla J_k(x^k)\rangle},
\end{equation}
for $M_k=\max\{M_f, \sqrt{\sigma_k}M_\psi \}$ via a backtracking or bisection approach.
Then, after
\begin{align*}
    T_k = \left\lceil \log_2 \frac{\ln  \Big(\tfrac{\sigma_k}{M_k} \sqrt{c(\sigma_k) + \tfrac{1}{\sigma_k}} \Big) + \max\Big\{ \tfrac{1}{2}\ln\tfrac{1}{2 \rho_k B_k(s^k)},  \ln(3M_\psi )\Big\}}{\ln 2} \right\rceil 
\end{align*}
pure Newton steps on $J_k$ starting from $x^k$, the stopping criterion~\cref{eq:error_bound_solodov_alm} is satisfied.
\end{proposition}
\begin{proof}
By \cref{thm:sc}, $J_k$ is strongly self-concordant on $\dom \psi=\dom J_k$ with constant $M_k=\max\{M_f, \sqrt{\sigma_k} M_\psi\}$. 
	Since $s^{k,0}=x^k \in \dom J_k$ we have by the choice of $\sigma_k$
	\begin{align} \label{eq:bound_first}
	\lambda(s^{k,0}) &=M_k\sqrt{\langle \nabla J_k(s^{k,0}), \nabla^2 J_k(s^{k,0})^{-1} \nabla J_k(s^{k,0})\rangle}\notag \\
	&\leq M_k\sqrt{\langle \nabla J_k(s^{k,0}), \sigma_k \nabla^2 \psi(s^{k,0})^{-1} \nabla J_k(s^{k,0})\rangle} < \tfrac{1}{4} < 2-\sqrt{3} < 1,
	\end{align}
	where the first inequality follows by the fact that $\nabla^2 J_k(s) \succeq \frac{1}{\sigma_k} \nabla^2 \psi(s) \succ 0$ and hence $\nabla^2 J_k(s)^{-1} \preceq \sigma_k \nabla^2 \psi(s)^{-1}$.
	Thus by \cite[Theorem 2.2.2(ii)]{nesterov1994interior} we have for $s^{k,t+1}=s^{k,t}-\nabla^2 J_k(s^{k,t})^{-1}\nabla J_k(s^{k,t})$ that
	\begin{align} \label{eq:quadratic_decrease_lambda}
	\dom J_k \ni \lambda (s^{k,t+1}) \leq \tfrac{1}{(1-\lambda(s^{k,t}))^2} \lambda^2(s^{k,t}) \leq 2 \lambda^2(s^{k,t}),
	\end{align}
	where the last inequality holds in view of \cref{eq:bound_first}.
	After $T_k$ iterations we have that
	\begin{align}\label{eq:bound_tk}
	\lambda(s^{k, T_k}) \leq \tfrac{1}{2} (2 \lambda(s^{k,0}))^{2^{T_k}} < \tfrac{1}{2} (\tfrac{1}{2})^{2^{T_k}}.
	\end{align}
By assumption $\nabla^2 F_k(s) \preceq c(\sigma_k) \nabla^2 \psi(s)$ for all $s \in \dom \psi$ and so $\nabla^2 J_k(s)\preceq (c(\sigma_k) +  \tfrac{1}{\sigma_k})\nabla^2 \psi(s)$. Inverting the matrix inequality yields for $s^{k,T_k}$ the bound
\begin{align} \label{eq:h_to_J_hessian_bound}
\nabla^2 \psi(s^{k,T_k})^{-1} \preceq (\tfrac{1}{\sigma_k}+c(\sigma_k))\nabla^2 J_k(s^{k,T_k})^{-1},
\end{align}
and so
\begin{align} \label{eq:hessian_sc_proof}
\sigma_k^2\langle \nabla J_k(s^{k,T_k}), \nabla^2 \psi(s^{k,T_k})^{-1} \nabla J_k(s^{k,T_k}) \rangle &\leq \tfrac{\sigma_k^2}{M_k^2}(c(\sigma_k) + \tfrac{1}{\sigma_k}) \lambda^2(s^{k,T_k}).
\end{align}
Next we show that
\begin{align} \label{eq:bound19}
\tfrac{\sigma_k^2}{M_k^2}(c(\sigma_k) + \tfrac{1}{\sigma_k}) \lambda^2(s^{k,T_k}) < \tfrac{1}{(3M_\psi)^2}.
\end{align}
In light of \cref{eq:bound_tk} this is implied if:
\begin{align}
(\tfrac{1}{2})^{2^{T_k}} < \frac{2}{3M_\psi \tfrac{\sigma_k}{M_k} \sqrt{\tfrac{1}{\sigma_k} + c(\sigma_k)}}.
\end{align}
Taking logarithms on both sides, multiplying by $-1$ and taking the logarithm with base $2$ this is implied by
\begin{align}
T_k  > \log_2 \frac{\ln \Big( \tfrac{\sigma_k}{M_k} \sqrt{\tfrac{1}{\sigma_k} + c(\sigma_k)} \Big) +\ln (3M_\psi)}{\ln 2}
\end{align}
which holds true by assumption and hence \cref{eq:bound19} is valid and via \cref{eq:hessian_sc_proof}, $\langle e, \nabla^2 \psi(x)^{-1} e \rangle < \tfrac{1}{9M_\psi^2}$ for $e:=\sigma_k \nabla J_k(s^{k,T_k})$ and $x=s^{k,T_k}$. Invoking \cref{thm:self_concordance_stopping} we obtain
\begin{align}
D_\psi(x^{+}(s^{k,T_k}), s^{k,T_k})&\leq 2\sigma_k^2 \langle \nabla J_k(s^{k,T_k}), \nabla^2 \psi(s^{k,T_k})^{-1} \nabla J_k(s^{k,T_k}) \rangle \notag \\
&=2\tfrac{\sigma_k^2}{M_k^2}(c(\sigma_k) + \tfrac{1}{\sigma_k})\lambda^2(s^{k,T_k}).
\end{align}
Thus the stopping criterion~\cref{eq:error_bound_solodov_alm} is guaranteed to hold if
\begin{align}
2\tfrac{\sigma_k^2}{M_k^2}(c(\sigma_k) + \tfrac{1}{\sigma_k})\lambda^2(s^{k,T_k}) \leq \rho_k(D_\psi(s^k, x^k) + D_\phi(y^{+}(s^k), y^k)) = \rho_k B_k(s^k).
\end{align}
In light of \cref{eq:bound_tk} this is implied if
\begin{align}
(\tfrac{1}{2})^{2^{T_k}} \leq \frac{\sqrt{2} \sqrt{\rho_k}\sqrt{ B_k(s^k) }}{\tfrac{\sigma_k}{M_k} \sqrt{c(\sigma_k) + \tfrac{1}{\sigma_k}}}.
\end{align}
Taking logarithms on both sides, multiplying by $-1$ and taking $\log_2$ this is implied by
\[
T_k \geq\log_2 \frac{\ln \Big(\tfrac{\sigma_k}{M_k} \sqrt{c(\sigma_k) + \tfrac{1}{\sigma_k}} \Big) + \tfrac{1}{2}\ln\tfrac{1}{2 \rho_k B_k(s^k)}}{\ln 2},
\]
which is valid by assumption.
\end{proof}

To conclude we revisit \cref{ex:qp} and provide a choice for $\psi$ that encodes the box constraints and the corresponding constant $c(\sigma_k)$
\begin{example}
In the situation of \cref{ex:qp} choose $\psi(x)=\frac{1}{2}\|x\|^2-\sum_{i=1}^n \ln(u_i-x_i) + \ln(x_i-l_i)$. Then $\psi$ is Legendre and self-concordant on $\dom \psi = \intr \dom f$ with constant $M_\psi=1$. Furthermore, $\nabla^2 F_k(s)=\sigma_k A^\top A + W^\top W \preceq c(\sigma_k) \nabla^2 \psi(s)$ for all $s$ and $c(\sigma_k):=\sigma_k \|A\|^2+\|W\|^2$.
\end{example}

\begin{remark}[Relation to interior-point proximal multiplier methods]
In the quadratic programming setting with box or conic constraints,
the above construction is closely related to the interior-point
proximal multiplier methods of Pougkakiotis and Gondzio
\cite{pougkakiotis2021interior,pougkakiotis2022interior}.
Their approach combines classical (Euclidean) proximal augmented Lagrangian updates
with barrier path-following and Newton steps, resulting in a scheme with strong global convergence and explicit
complexity guarantees in structured settings such as QP and SDP.

The present framework differs in scope and perspective.
By formulating the method as a Bregman proximal point algorithm
applied to the KKT operator, we obtain a unified operator-theoretic
view that extends beyond quadratic programming to general
convex composite problems and nonlinear mappings.
However, this generality comes at the price that the strong
polynomial-type complexity guarantees available in structured
interior-point settings cannot be recovered in full generality.
Instead, our analysis relies on metric subregularity and
local equivalence of Bregman and Euclidean geometries to derive
outer-loop rates and joint complexity bounds.
\end{remark}

\section{Conclusion} \label{sec:conclusion}
In this work we studied Bregman proximal augmented Lagrangian methods equipped with second-order oracles for convex convex-composite optimization problems. By formulating the scheme as a Bregman proximal point algorithm applied to the KKT operator, we obtained a unified operator-theoretic perspective that connects proximal augmented Lagrangian methods with classical Lagrange--Newton and primal-dual interior-point approaches. This viewpoint clarifies the role of regularization and explains how smooth, generalized self-concordant subproblems arise naturally after marginalizing out the multiplier variables.

On the theoretical side, we refined the convergence analysis of inexact Bregman proximal point methods. In particular, we established asymptotic dual convergence and ergodic primal convergence without imposing interiority of the solution. Using metric subregularity and Hoffman-type error bounds for the KKT operator, we derived explicit outer-loop rates. When combined with a path-following Newton oracle for the inner subproblems and a careful alignment of the relative inexactness criterion with the Newton decrement, this yields a local joint complexity bound for the overall method.

The proposed framework encompasses several important algorithms as special cases, including proximal variants of the exponential multiplier method and interior-point proximal augmented Lagrangian schemes. We hope that the operator-theoretic and geometric viewpoint developed here contributes to a clearer structural understanding of modern Lagrange--Newton methods and stimulates further research on non-Euclidean proximal regularization techniques.

\subsection*{Acknowledgements}
Parts of this work were developed interactively using the large language model \textsc{ChatGPT-5.2 Pro} in a propose-verify refinement loop. In particular, the gap-function analysis for the Bregman proximal point algorithm and the bridge lemma connecting the Newton decrement with the Bregman distance were initially suggested in this interactive setting.

All suggestions produced by the LLM were treated as unverified hypotheses and were carefully checked by the authors. The final statements, proofs, and manuscript text were independently validated and written by the human authors.
\printbibliography

@article{pougkakiotis2022interior,
  title={An interior point-proximal method of multipliers for linear positive semi-definite programming},
  author={Pougkakiotis, Spyridon and Gondzio, Jacek},
  journal={Journal of Optimization Theory and Applications},
  volume={192},
  number={1},
  pages={97--129},
  year={2022},
  publisher={Springer}
}

@article{pougkakiotis2021interior,
  title={An interior point-proximal method of multipliers for convex quadratic programming},
  author={Pougkakiotis, Spyridon and Gondzio, Jacek},
  journal={Computational Optimization and Applications},
  volume={78},
  number={2},
  pages={307--351},
  year={2021},
  publisher={Springer}
}

@inproceedings{schwan2023piqp,
  title={PIQP: A proximal interior-point quadratic programming solver},
  author={Schwan, Roland and Jiang, Yuning and Kuhn, Daniel and Jones, Colin N},
  booktitle={2023 62nd IEEE Conference on Decision and Control (CDC)},
  pages={1088--1093},
  year={2023},
  organization={IEEE}
}

@inproceedings{hermans2019qpalm,
  title={QPALM: a Newton-type proximal augmented Lagrangian method for quadratic programs},
  author={Hermans, Ben and Themelis, Andreas and Patrinos, Panagiotis},
  booktitle={2019 IEEE 58th Conference on Decision and Control (CDC)},
  pages={4325--4330},
  year={2019},
  organization={IEEE}
}

@article{luque1984asymptotic,
  title={Asymptotic convergence analysis of the proximal point algorithm},
  author={Luque, Fernando Javier},
  journal={SIAM Journal on Control and Optimization},
  volume={22},
  number={2},
  pages={277--293},
  year={1984},
  publisher={SIAM}
}

@article{doikov2025minimizing,
  title={Minimizing Quasi-Self-Concordant Functions by Gradient Regularization of Newton Method: N. Doikov},
  author={Doikov, Nikita},
  journal={Mathematical Programming},
  pages={1--39},
  year={2025},
  publisher={Springer}
}

@article{dvurechensky2019generalized,
  title={Generalized self-concordant hessian-barrier algorithms},
  author={Dvurechensky, Pavel and Staudigl, Mathias and Uribe, Cesar A},
  journal={arXiv preprint arXiv:1911.01522},
  year={2019}
}

@article{sun2019generalized,
  title={Generalized self-concordant functions: a recipe for newton-type methods},
  author={Sun, Tianxiao and Tran-Dinh, Quoc},
  journal={Mathematical Programming},
  volume={178},
  number={1},
  pages={145--213},
  year={2019},
  publisher={Springer}
}

@article{bach2010self,
author = {Francis Bach},
title = {{Self-concordant analysis for logistic regression}},
volume = {4},
journal = {Electronic Journal of Statistics},
publisher = {Institute of Mathematical Statistics and Bernoulli Society},
pages = {384 -- 414},
year = {2010},
doi = {10.1214/09-EJS521},
URL = {https://doi.org/10.1214/09-EJS521}
}

@article{rockafellar2021advances,
  title={Advances in convergence and scope of the proximal point algorithm},
  author={Rockafellar, R Tyrrell},
  journal={J. Nonlinear and Convex Analysis},
  volume={22},
  pages={2347--2375},
  year={2021}
}

@article{rockafellar1976augmented,
  title={Augmented Lagrangians and applications of the proximal point algorithm in convex programming},
  author={Rockafellar, R Tyrrell},
  journal={Mathematics of operations research},
  volume={1},
  number={2},
  pages={97--116},
  year={1976},
  publisher={INFORMS}
}

@article{xu2021iteration,
  title={Iteration complexity of inexact augmented Lagrangian methods for constrained convex programming},
  author={Xu, Yangyang},
  journal={Mathematical Programming},
  volume={185},
  pages={199--244},
  year={2021},
  publisher={Springer}
}

@article{burachik1998generalized,
  title={A generalized proximal point algorithm for the variational inequality problem in a Hilbert space},
  author={Burachik, Regina S and Iusem, Alfredo N},
  journal={SIAM journal on Optimization},
  volume={8},
  number={1},
  pages={197--216},
  year={1998},
  publisher={SIAM}
}

@article{yan2020bregman,
  title={Bregman Augmented Lagrangian and its acceleration},
  author={Yan, Shen and He, Niao},
  journal={arXiv preprint arXiv:2002.06315},
  year={2020}
}

@article{solodov2000inexact,
  title={An inexact hybrid generalized proximal point algorithm and some new results on the theory of Bregman functions},
  author={Solodov, Mikhail V and Svaiter, Benar Fux},
  journal={Mathematics of Operations Research},
  volume={25},
  number={2},
  pages={214--230},
  year={2000},
  publisher={INFORMS}
}

@article{laude2025anisotropic,
  title={Anisotropic proximal gradient},
  author={Laude, Emanuel and Patrinos, Panagiotis},
  journal={Mathematical Programming},
  pages={1--45},
  year={2025},
  publisher={Springer}
}

@article{eckstein1998approximate,
  title={Approximate iterations in Bregman-function-based proximal algorithms},
  author={Eckstein, Jonathan},
  journal={Mathematical programming},
  volume={83},
  pages={113--123},
  year={1998},
  publisher={Springer}
}

@article{eckstein2003practical,
  title={A practical general approximation criterion for methods of multipliers based on Bregman distances},
  author={Eckstein, Jonathan},
  journal={Mathematical Programming},
  volume={96},
  pages={61--86},
  year={2003},
  publisher={Springer}
}

@book{nesterov2018lectures,
  title={Lectures on Convex Optimization},
  author={Nesterov, Yurii},
  volume={137},
  year={2018},
  publisher={Springer}
}

@article{nesterov1994interior,
  title={Interior-Point Polynomial Algorithms in Convex Programming},
  author={Nesterov, Yurii and Nemirovskii, Arkadii},
  year={1994},
  publisher={Society for Industrial and Applied Mathematics}
}

@article{rockafellar1970maximality,
  title={On the maximality of sums of nonlinear monotone operators},
  author={Rockafellar, R Tyrrell},
  journal={Transactions of the American Mathematical Society},
  volume={149},
  number={1},
  pages={75--88},
  year={1970},
  publisher={JSTOR}
}

@article{chen1993convergence,
  title={Convergence analysis of a proximal-like minimization algorithm using Bregman functions},
  author={Chen, Gong and Teboulle, Marc},
  journal={SIAM Journal on Optimization},
  volume={3},
  number={3},
  pages={538--543},
  year={1993},
  publisher={SIAM}
}

@article{bauschke1997legendre,
	author = {Bauschke, Heinz H. and Borwein, Jonathan M.},
	journal = {Journal of Convex Analysis},
	number = {1},
	pages = {27-67},
	publisher = {Heldermann Verlag},
	title = {Legendre functions and the method of random {B}regman projections},
	volume = {4},
	year = {1997}}

@book{BaCo110,
	author = {Bauschke, Heinz H. and Combettes, Patrick L.},
	publisher = {Springer Science \& Business Media},
	title = {Convex Analysis and Monotone Operator Theory in {H}ilbert Spaces},
	year = {2011}}

@book{RoWe98,
	address = {New York},
	author = {Rockafellar, Ralph T. and Wets, Roger J.B.},
	publisher = {Springer},
	title = {Variational Analysis},
	year = {1998}}

@phdthesis{laude2021lower,
	author = {Laude, Emanuel},
	school = {Technical University of Munich},
	title = {Lower envelopes and lifting for structured nonconvex optimization},
	year = {2021}}

@article{chen1997gap,
  title={On gap functions and duality of variational inequality problems},
  author={Chen, GY and Goh, CJ and Yang, XQ},
  journal={Journal of Mathematical Analysis and Applications},
  volume={214},
  number={2},
  pages={658--673},
  year={1997},
  publisher={Elsevier}
}

@book{Roc70,
	address = {New Jersey},
	author = {Rockafellar, R. Tyrrell},
	publisher = {Princeton University Press},
	title = {Convex Analysis},
	year = {1970}}

@article{tseng1993convergence,
  title={On the convergence of the exponential multiplier method for convex programming},
  author={Tseng, Paul and Bertsekas, Dimitri P},
  journal={Mathematical programming},
  volume={60},
  number={1-3},
  pages={1--19},
  year={1993},
  publisher={Springer}
}

@article{censor1992proximal,
	author = {Censor, Y. and Zenios, S. A.},
	journal = {Journal of Optimization Theory and Applications},
	number = {3},
	pages = {451--464},
	publisher = {Springer},
	title = {Proximal minimization algorithm with {D}-functions},
	volume = {73},
	year = {1992}}

@article{bauschke2000dykstras,
	author = {Bauschke, Heinz H and Lewis, Adrian S},
	journal = {Optimization},
	number = {4},
	pages = {409--427},
	publisher = {Taylor \& Francis},
	title = {Dykstra's algorithm with {B}regman projections: A convergence proof},
	volume = {48},
	year = {2000}}

@article{Eckstein93,
	author = {J. Eckstein},
	journal = {Mathematics of Operations Research},
	number = {1},
	pages = {202--226},
	title = {Nonlinear Proximal Point Algorithms Using {Bregman} Functions, with Applications to Convex Programming},
	volume = {18},
	year = {1993}}
\end{document}